\documentclass[11pt,reqno,tbtags]{amsart}
\allowdisplaybreaks
\numberwithin{equation}{section}
  
\usepackage[foot]{amsaddr}
\usepackage{amsthm,amssymb,amsfonts,amsmath}
\usepackage{mathrsfs}
\usepackage{mathtools}
\usepackage[numbers, sort&compress]{natbib}
\usepackage{subfigure,graphicx}
\usepackage{vmargin, enumerate}
\usepackage{setspace}
\usepackage{paralist}
\usepackage[pagebackref=true]{hyperref}
\usepackage{color}
\usepackage{stmaryrd} 
\usepackage{tikz}
\usetikzlibrary{arrows}
\usetikzlibrary{decorations.markings}

\newcommand{\R}{\mathbb{R}}

\newcommand{\N}{{\mathbb N}}

\newcommand{\dS}{\mathbb{S}}
\newcommand{\K}{\mathbb{K}}

\newcommand{\E}[1]{{\mathbb E}\left[#1\right]}

\newcommand{\p}[1]{{\mathbb P}\left(#1\right)}

\newcommand{\I}[1]{{\mathbf 1}_{[#1]}}

\usepackage[normalem]{ulem}

\newtheorem{thm}{Theorem}[section]
\newtheorem{lem}[thm]{Lemma}
\newtheorem{prop}[thm]{Proposition}
\newtheorem{cor}[thm]{Corollary}

\newtheorem{fact}[thm]{Fact}

\newcommand\cB{\mathcal B}

\newcommand\cE{\mathcal E}
\newcommand\cF{\mathcal F}
\newcommand\cG{\mathcal G}

\newcommand\cO{\mathcal O}
\newcommand\cP{\mathcal P}
\newcommand\cQ{\mathcal Q}
\newcommand\cR{{\mathcal R}}
\newcommand\cS{{\mathcal S}}
\newcommand\cT{{\mathcal T}}

\newcommand{\rB}{\mathrm{B}}

\newcommand{\rd}{\mathrm{d}}

\newcommand{\rG}{\mathrm{G}} 
\newcommand{\rH}{\mathrm{H}}

\newcommand{\rM}{\mathrm{M}}

\newcommand{\rP}{\mathrm{P}} 
\newcommand{\rQ}{\mathrm{Q}} 
\newcommand{\rR}{\mathrm{R}}

\newcommand{\rW}{\mathrm{W}} 
\newcommand{\rX}{\mathrm{X}}

\newcommand{\rn}{\mathrm{n}}

\newcommand{\bB}{\mathbf{B}}

\newcommand{\be}{\mathbf{e}}

\newcommand{\bL}{\mathbf{L}}

\newcommand{\bQ}{\mathbf{Q}} 
\newcommand{\bR}{\mathbf{R}}

\newcommand{\bV}{\mathbf{V}} 
\newcommand{\bW}{\mathbf{W}} 
\newcommand{\bX}{\mathbf{X}} 
\newcommand{\bY}{\mathbf{Y}} 
\newcommand{\bZ}{\mathbf{Z}} 
 
\newcommand{\bm}{\mathbf{m}} 
\newcommand{\bn}{\mathbf{n}} 

\DeclareMathAlphabet\mathbfcal{OMS}{cmsy}{b}{n}

\newcommand{\refT}[1]{Theorem~\ref{#1}}
\newcommand{\refC}[1]{Corollary~\ref{#1}}

\newcommand{\refL}[1]{Lemma~\ref{#1}}

\newcommand{\refS}[1]{Section~\ref{#1}}
\newcommand{\refP}[1]{Proposition~\ref{#1}}

\newcommand{\refFt}[1]{Fact~\ref{#1}}
\newcommand{\refApp}[1]{Appendix~\ref{#1}}

\providecommand{\veps}{}
\renewcommand{\veps}{\varepsilon}
\providecommand{\ora}[1]{}
\renewcommand{\ora}[1]{\overrightarrow{#1}}
\newcommand{\diam}{\mathrm{diam}}

\newcommand{\dghp}{\ensuremath{\rd_{\mathrm{GHP}}}}

\newcommand{\dlghp}{\ensuremath{\rd_{\mathrm{LGHP}}}}

\newcommand{\eqdist}{\ensuremath{\stackrel{\mathrm{d}}{=}}}
\newcommand{\convdist}{\ensuremath{\stackrel{\mathrm{d}}{\to}}}

\newcommand{\convp}{\ensuremath{\stackrel{\mathrm{p}}{\to}}}

\hypersetup{
    bookmarks=true,         
    unicode=false,          
    pdftoolbar=true,        
    pdfmenubar=true,        
    pdffitwindow=true,      
    pdftitle={My title},    
    pdfauthor={Author},     
    pdfsubject={Subject},   
    pdfnewwindow=true,      
    pdfkeywords={keywords}, 
    colorlinks=true,        
    linkcolor=blue,         
    citecolor=blue,         
    filecolor=blue,         
    urlcolor=blue           
}

\reversemarginpar
\definecolor{clou}{rgb}{0.8,0.25,0.5125}

\begingroup
  \divide\count255 by 60
  \multiply\count255 by -60
  \advance\count255 by \time
  \ifnum \count255 < 10 \xdef\oclock{\the\count1:0\the\count255}
  \else\xdef\oclock{\the\count1:\the\count255}\fi
\endgroup

\DeclareRobustCommand{\SkipTocEntry}[5]{}

\begin{document}

\title{The Brownian Plane with Minimal Neck Baby Universe}
\author{Yuting Wen}
\email{\href{mailto:yutingyw@gmail.com}{\tt yutingyw@gmail.com}}

\address{School of Mathematics and Statistics, University of Melbourne, Parkville, VIC, Australia}

\date{\today}
\keywords{Brownian map, Brownian plane, scaling limit, local Gromov-Hausdorff-Prokhorov topology, quadrangulation, connectivity, core, condensation, random allocation.}
\subjclass[2010]{60C05,05C10} 

\begin{abstract} 
For each $n\in\mathbb{N}$, let $\mathbf{Q}_n$ be a uniform rooted quadrangulation, endowed with an appropriate measure, of size $n$ conditioned to have $r(n)$ vertices in its root block. We prove that for a suitable function $r(n)$, after rescaling graph distance by $\left(\frac{21}{40\cdot r(n)}\right)^{1/4}$, $\mathbf{Q}_n$ converges to a random pointed non-compact metric measure space $\boldsymbol{ \mathcal{S}}$, in the local Gromov-Hausdorff-Prokhorov topology. The space $\boldsymbol{ \mathcal{S}}$ is built by identifying a uniform point of the Brownian map with the distinguished point of the Brownian plane. 
\end{abstract}

\maketitle

\section{Introduction}

The scaling limit of large random planar maps has been a focal point of probability research in the recent decade. \citet{LG} and \citet{Mi} independently established that the Brownian map is the scaling limit of several important families of planar maps; \citet{BJM} and \citet{Ab}, respectively, proved that general and bipartite planar maps with a fixed number of edges converge to the Brownian map after rescaling; \citet{ABA} and \citet{ABW}, respectively, showed that simple quadrangulations and $2$-connected quadrangulations also rescale to the same limit object. 

\citet{CLG} defined an infinite-volume version of the Brownian map, called the Brownian plane. It shares numerous similarities with the Brownian map, but additionally possesses the scaling invariance property. 

In this paper, we show a convergence towards a pointed non-compact metric measure space $\mathbfcal{S}$, obtained by identifying a random point of the Brownian map and the distinguished point of the Brownian plane; see Appendix~\ref{sec:minibus} for details. This random geometry structure provides a probabilistic model of the so-called minimal neck baby universe in $2$-dimensional quantum gravity; see \citet{JM}. Motivated by this notion, we call $\mathbfcal{S}$ the {\em Brownian plane with minimal neck baby universe} ({\em minbus} in the literature). We show that $\mathbfcal{S}$ is the limit of rescaled random quadrangulations conditioned on having an exceptionally large root block.

The result relies upon both the convergence of uniform quadrangulations towards the Brownian plane \cite{CLG}, and the convergence of uniform $2$-connected quadrangulations to the Brownian map \cite{ABW}. The main steps of the proof are as follows. First, we show that the sizes of submaps pendant to the root block have an asymptotically stable distribution. Second, we deduce asymptotics for occupancy in a random allocation model with a varying balls-to-boxes ratio. Third, we establish a bound for the number of pendant submaps of the root block, which allows us to apply the occupancy bounds to uniformly control the sizes of pendant submaps. This entails us to show that the pendant submaps act as uniformly asymptotically negligible ``decorations'' which do not affect the scaling limit.

\subsection{Terminologies}\label{sec:graphintro}

In this paper, graphs are allowed to have multiple edges. Fix a graph $G$. Write $d_G$ for the graph distance on $G$. Denote by $v(G)$ and $e(G)$, respectively, the vertex set and the edge set of $G$. The {\em size} of $G$, $|v(G)|$, is the number of vertices. For $V\subset v(G)$, write $G[V]$ for the subgraph of $G$ induced by $V$, and let $G-V =G[v(G)\backslash V]$. We say that $G$ is {\em $2$-connected} if the removal of any vertex does not disconnect $G$. A {\em rooted graph} is a pair $\rG=(G,e)$ where $e$ is an oriented edge of $G$. If $G$ is a single vertex, say $v$, then we say that $(G,v)$ is the corresponding rooted graph. 

A {\em (planar) map} is a finite planar graph properly embedded on the $2$-sphere $\dS^2$, considered up to orientation-preserving homeomorphism. A {\em submap} is an embedded subgraph. In the case that only the metric structure of a map is relevant, we may call it a graph.

A cycle $C$ in a map $M$ is {\em facial} if at least one connected component of $\mathbb{S}^2\backslash C$ contains neither vertices nor edges of $M$. Write $M^\circ$ for the map obtained from $M$ by collapsing each facial $2$-cycle into an edge.

Fix a rooted map $\rM=(M,e)$. We may define a total order $<_{\rM}$ on $v(M)$ as follows. Given that $e$ orients from $u$ to $v$, we list the vertices of $\rM$ as $u_1=u,u_2=v,u_3,...,u_{|v(M)|}$ according to their order of exploration by a breadth-first search, using the clockwise order of edges around each vertex to determine exploration priority. Furthermore, we define a total order $\prec_{\rM}$ on the set of oriented edges of $\rM$ by setting the root edge to be $\prec_\rM$-minimal and letting $e\prec_{\rM} e'$ if, given that $e$ orients from $u_i$ to $u_j$ and $e'$ from $u_{i'}$ to $u_{j'}$, either (1) $u_i$ was explored before $u_{i'}$, or (2) $i=i'$ and $e$ has higher exploration priority than $e'$. 

A {\em quadrangulation} is a connected, bipartite map where each face has degree $4$. Fix a rooted quadrangulation $\rQ=(Q,e)$. Given an edge $e' \in e(Q)$, let $B_{e'}\subset v(Q)$ be maximal subject to the constraints that $Q[B_{e'}]$ is $2$-connected and that the endpoints of $e'$ belong to $B_{e'}$; we call $Q[B_{e'}]^\circ$ a {\em block} in $Q$. In particular, for $e'=e$, we call $\left(Q[B_e],e\right)$ the {\em pre-root-block} of $\rQ$, and call $R(\rQ)\coloneqq Q[B_e]^\circ$ the {\em root block} of $\rQ$. Note that the faces of the pre-root-block consist in quadrangles and $2$-angles.

Block decomposition and block sizes have been studied for a long time; see, for example, the papers by \citet{T}, \citet{GW}, \citet{BFSS}. Here we use a variant of block decomposition, described below.\footnote{The map decomposition here is different from the block decomposition given by \citet{A}, which decomposes a map in terms of a block tree where each vertex of the tree represents a block.} Write $\cF = \cF(\rQ)$ for the set of faces enclosed by the facial $2$-cycles in the pre-root-block of $\rQ$. Assume that $\cF\neq \emptyset$, which is always satisfied for the quadrangulations we consider in the sequel. List the elements of $\cF$ as $f(1),\ldots, f(|\cF|)$, in the $\prec_\rQ$-order of their incident edges, called the {\em canonical order} on $\cF$. For $i\in \{1,\ldots,|\cF|\}$, write $P_i=P_i(\rQ)$ for the unique maximal connected non-singleton submap of $\rQ$ lying in $f(i)$ and containing no edge of the facial $2$-cycle that encloses $f(i)$, let $e_i$ be the $\prec_{\rQ}$-minimal edge in $P_i$, and let $\rP_i = (P_i,e_i)$ be the rooted submap. See Figure~\ref{con}.

\begin{figure}[htb]
\centering
\subfigure[rooted quadrangulation $\rQ$]{
\begin{tikzpicture}[scale=0.5]
\begin{scope}[decoration={
    markings,
    mark=at position 0.5 with {\arrow{>}}}
    ]
	\tikzstyle{main}=[circle,inner sep=0pt,minimum size=5pt,fill=black!70];

	\node[main] (a) at (-2,4) {};
	\node[main] (b) at (2,4){};
	\node[main] (c) at (4,4){};
	\node[main] (d) at (8,4) {};
	\node[main] (b1) at (1,4.5){};
	\node[main] (b2) at (0,4.9){};
	\node[main] (b3) at (-1,5.3){};
	\node[main] (a1) at (-1,3){};
	\node[main] (e) at (3,4) {};
	\node[main] (f) at (4,9){};
	\draw (a) to[bend left=45] (f) to[bend left=45] (d);
	\draw (b) to (e) to (d);
	\draw (a) to[bend left=40]  (1,7) to[bend left=40] (c);
	\draw (a) to[bend right=40]  (1,1) to[bend right=40] (c);
	\draw (c) to  (d);
	\draw[postaction={decorate}] (a) to[bend left=40] (0,6) to[bend left=40]   (b);
	\draw (b) to (b1) to (b2) to (b3);
	\draw (b) to[bend left=30] (b3);
	\draw (b) to[bend right=30] (b3);

	\draw (a) to[bend right=30]   (b);
	\draw (a) to (a1);
	\draw (a) to[bend right=40] (0,2) to[bend right=40]  (b);
	
	\node[main] (d1) at (7,3){};
	\draw (d) to (d1);
	\node[main] (c1) at (5,5){};
	\node[main] (c2) at (6,6){};
	\draw (c) to (c1);
	\draw (c) to[bend left=45] (6,6.5) to[bend left=45]    (d);
	\draw (c) to[bend right=45] (6,1.5) to[bend right=45]  (d);
	\draw (c) to[bend left=30] (c2);
	\draw (c) to[bend right=30]  (c2);	
\end{scope}
\end{tikzpicture}}
\subfigure[pre-root-block $\rB$ of $\rQ$]{
\begin{tikzpicture}[scale=0.5]
\begin{scope}[decoration={
    markings,
    mark=at position 0.5 with {\arrow{>}}}
    ]
	\tikzstyle{main}=[circle,inner sep=0pt,minimum size=5pt,fill=black!70];

	\node[main] (a) at (-2,4) {};
	\node[main] (b) at (2,4){};
	\node[main] (c) at (4,4){};
	\node[main] (d) at (8,4) {};
	\node[main] (e) at (3,4) {};
	
	\node (a1) at (-1.5,4.2) {};
	\node (a2) at (-1.5,3.4) {};
	\node (b1) at (1.5,4.2) {};
	\node (b2) at (1.5,3.4) {};
	\node (c1) at (4.5,4.2) {};
	\node (c2) at (4.5,3.7) {};
	\node (d1) at (7.5,4.2) {};
	\node (d2) at (7.5,3.7) {};
	\node[main] (f) at (4,9){};
	\node (f1) at (0,4.7) {$f(1)$};
	\node (f2) at (0,2.6) {$f(2)$};
	\node (f3) at (6,5.1) {$f(3)$};
	\node (f4) at (6,2.6) {$f(4)$};

	\draw (a) to[bend left=45] (f) to[bend left=45] (d);	
	\draw (b) to (e) to (d);	
	\draw (a) to[bend left=40]  (1,7) to[bend left=40] (c);
	\draw (a) to[bend right=40]  (1,1) to[bend right=40] (c);
	\draw (c) to  (d);
	\draw[postaction={decorate}] (a) to[bend left=40] (0,6) to[bend left=40]   (b);

	\draw (a) to[bend right=30]   (b);
	\draw (a) to[bend right=40] (0,2) to[bend right=40]  (b);
	\draw (c) to[bend left=45] (6,6.5) to[bend left=45]    (d);
	\draw (c) to[bend right=45] (6,1.5) to[bend right=45]  (d);
\end{scope}
\end{tikzpicture}}
\\
\subfigure[submaps $(\rP_i:1\le i\le |\cF|)$]{
\begin{tikzpicture}[scale=0.5]

\begin{scope}[decoration={
    markings,
    mark=at position 0.5 with {\arrow{>}}}
    ]
	\tikzstyle{main}=[circle,inner sep=0pt,minimum size=5pt,fill=black!70];

\node[main,label={[xshift=0cm,yshift=-1cm]$\rP_1$}] (a) at (-5,-1) {};
\node[main] (a1) at (-5,0){};
\node[main] (a2) at (-5,1){};
\node[main] (a3) at (-5,2){};
\draw[postaction={decorate}] (a) to[bend left=50] node[left] {$e_1$} (a3);
\draw (a) to (a1) to (a2) to (a3);
\draw (a) to[bend right=50] (a3);

\node[main,label={[xshift=0cm,yshift=-1cm]$\rP_2$}] (b) at (-2,-1) {};	
\node[main] (b1) at (-2,2){};
\draw[postaction={decorate}] (b) to node[left] {$e_2$} (b1);
	
\node[main,label={[xshift=0cm,yshift=-1cm]$\rP_3$}] (c) at (1,-1) {};	
\node[main] (c1) at (1,0.5){};
\node[main] (c2) at (1,2){};
\draw[postaction={decorate}] (c) to[bend left=50] node[left] {$e_3$} (c2);
\draw (c) to (c1);
\draw (c) to[bend right=50] (c2);

\node[main,label={[xshift=0cm,yshift=-1cm]$\rP_4$}] (d) at (4,-1) {};	
\node[main] (d1) at (4,2) {};
\draw[postaction={decorate}] (d) to node[left] {$e_4$} (d1);
\end{scope}
\end{tikzpicture}}	

\caption{\small}
\label{con}
\end{figure}

Next, let $L=L(\rQ)\in \{P_1,\ldots,P_{|\cF|}\}$ be the element with the largest size. If there are multiple elements of $\{P_1,\ldots,P_{|\cF|}\}$ of maximal size, we take $L$ to be the one which contains the $\prec_{\rQ}$-minimal edge. Furthermore, write 
\[
\rho_{\rQ} = v(R(\rQ)) \cap v(L(\rQ))
~\mbox{ and } ~
R^+(\rQ) = Q - \left(v\left(L(\rQ)\right)\setminus\{\rho_\rQ\}\right).
\]
In words, $\rho_\rQ$ is the unique vertex that connects the root block $R(\rQ)$ and the largest submap pendant to $R(\rQ)$, and $R^+(\rQ)$ consists of $R(\rQ)$ and all non-largest submaps pendant to it.

Given a set of graphs $\cG$, for $n\in \N$, let $\cG_n = \{G\in\cG:|v(G)|=n\}$. Write $\cQ$ and $\cR$ for the sets of connected and $2$-connected rooted quadrangulations, respectively. For all $r\in\N$ with $r\le n$, let
\[
\cQ_{n,r} = \left\{ \rQ\in\cQ_n: |v(R(\rQ))| = r\right\}.
\]

Given a finite set $\cG$, the notation $G\in_u\cG$ means that $G$ is chosen uniformly at random from $\cG$. We denote by $\convdist$ and $\convp$ convergence in distribution and in probability, respectively. Convergence and asymptotics statements are for $n\to\infty$, unless stated otherwise. When we say that a sequence $(E_n:n\in\N)$ of events occurs with high probability, we mean that $\p{E_n}\to1$. Write $\N=\{1,2,\ldots\}$ and $\N_{\ge0}=\{0,1,\ldots\}$. Finally, write GH(P) short for Gromov-Hausdorff(-Prokhorov).

\subsection{Convergence in the Local GHP Topology}\label{sec:main}
For the current subsection, a reference to \refS{sec:planeminbus} for the definition of $\mathbfcal{S}$ may be helpful, though the intuition of $\mathbfcal{S}$ presented above should be sufficient for the comprehension of the following contents. 

Since $\mathbfcal{S}$ contains the Brownian plane as a subspace, $\mathbfcal{S}$ is non-compact and the usual GH(P) topology is not suitable. We thus need the notion of local GHP topology, described below. See \cite{BBI,Mi, Mi09, LG,ABA} for details on the GH(P) topology, \citep[Section 1.2]{CLG} and \citep[Section 8.1]{BBI} for the pointed and local GH topologies, and \cite[Section 2.3]{ADH} for the local GHP topology. We call a metric space $(V,d)$ a {\em length space} if, for any $x,y\in V$, $d(x,y)$ equals the infimum of the lengths of continuous curves connecting $x$ and $y$; see \citep[Chapter 2]{BBI}. We call a metric space $(V,d)$ {\em boundedly compact} if all closed balls of finite radius are compact. A {\em pointed metric measure space} is a quadruple $(V,d,o,\nu)$, where $(V,d)$ is a metric space, $o\in V$ is called the {\em distinguished} point, and $\nu$ is a Borel measure on $(V,d)$. Given a pointed metric measure space $\bV=(V,d,o,\nu)$, for any $r\ge 0$, let $B_r = B_r(\bV)=\{w\in V: d(w,o)\le r\}$, and write $\bB_r(\bV)=\left(B_r,d,o,\nu\big\vert_{B_r}\right)$; we continue to use $d$ to denote the metric restricted to a subspace. Informally, a sequence of pointed boundedly compact measure length spaces $\left(\bV_n:n\in\N\right)$ converges to $\bV$ in the {\em local} GHP topology if for any $r\ge 0$, $\bB_r(\bV_n)$ converges to $\bB_r(\bV)$ in the {\em pointed} GHP topology.

A rooted graph $(G,e)$ is not a length space, but we may approximate it by a pointed boundedly compact length space where  the balls centred at the tail\footnote{Given an oriented edge $e$, if $e$ orients from $u$ to $v$, we call $u$ the tail of $e$ and $v$ the head.} of $e$ in $(G,e)$ and at the distinguished point in the approximating space are within GHP distance $1$. More precisely, we view each edge of $G$ as an isometric copy of the unit interval $[0,1]$. When we say that graphs converge in the pointed or local GHP topology, that is for their approximating spaces. Abusing notation, we continue to write $G$ for the resulting length space, and let $d_G$ be the intrinsic metric. 

Let $\mu_{G} = \sum_{v\in v(G)}\delta_v$ be the counting measure on $v(G)$.

\begin{thm}\label{thm2}
Let $r:\N\to\N$ be such that $r(n)> (\ln n)^{25}$ for all $n$ and $r(n) = o(n)$ as $n\to\infty$. Then for $\rQ_n\in_u \cQ_{n,r(n)}$, writing $k_n =\left(\frac{40\cdot r(n)}{21}\right)^{1/4}$, we have
\[
\left(\rQ_n,~ \frac{1}{k_n}\cdot d_{\rQ_n},~\rho_{\rQ_n},~ \frac{8}{9k_n^4}\cdot \mu_{L(\rQ_n)} + \frac{1}{|v(R^+(\rQ_n))|}\cdot \mu_{R^+(\rQ_n)}\right) \to \mathbfcal{S}
\]
in distribution for the local Gromov-Hausdorff-Prokhorov topology.
\end{thm}

The assumption $r(n)>(\ln n)^{25}$ is necessary due to (\ref{eq:assum}) but it is not optimal, while the assumption $r(n)=o(n)$ allows us to obtain a non-compact metric space in the limit. 

By assigning $0$ mass to components of $R^+(\rQ_n) - v(R(\rQ_n))$, we obtain a similar result.

\begin{thm}\label{thm1}
Let $r:\N\to\N$ be such that $r(n)> (\ln n)^{25}$ for all $n$ and $r(n) = o(n)$ as $n\to\infty$. Then for $\rQ_n\in_u \cQ_{n,r(n)}$, writing $k_n =\left(\frac{40\cdot r(n)}{21}\right)^{1/4}$, we have
\[
\left(\rQ_n,~ \frac{1}{k_n}\cdot d_{\rQ_n},~\rho_{\rQ_n},~ \frac{8}{9k_n^4}\cdot \mu_{L(\rQ_n)} + \frac{1}{r(n)}\cdot \mu_{R(\rQ_n)}\right) \to \mathbfcal{S}
\]
in distribution for the local Gromov-Hausdorff-Prokhorov topology.
\end{thm}

The proof of Theorem~\ref{thm1} is similar to but simpler than that of Theorem~\ref{thm2}, so we only provide a proof outline for Theorem~\ref{thm2}.

In the remainder of the paper, let $r:\N\to\N$ be such that $r(n)>(\ln n)^{25}$ for all $n$ and $r(n)=o(n)$. For all $n\in\N$, let $k_n =\left(\frac{40\cdot r(n)}{21}\right)^{1/4}$.

\subsection{Proof Outline for \refT{thm2}}\label{sec:outline}

Write $\mathbfcal{P}$ for the pointed measured Brownian plane and $\bm_\infty$ for the pointed measured Brownian map, both endowed with uniform measures; see the appendix for precise definitions.

In this subsection, for all $n\in\N$, let $\rQ_n\in_u\cQ_{n,r(n)}$. It is easily seen that $R_n \coloneqq R(\rQ_n)$ is a uniform $2$-connected quadrangulation with $r(n)$ vertices. Then by \citep[Theorem 1.1]{ABW}, $\widehat{\bR}_n \coloneqq \left(R_n,~\frac{1}{k_n}\cdot d_{R_n},~\rho_{\rQ_n},~\frac{1}{r(n)}\cdot \mu_{R_n}\right) \convdist \bm_\infty$ for the pointed GHP topology.

Write $R^+_n=R^+_n(\rQ_n)$. To establish an analogous convergence result for $R^+_n$, we show that components of $R^+_n - v(R_n)$ are uniformly asymptotically negligible. This is accomplished in two main steps. First, \refP{allocationprop2} and \refC{Rcor2} prove that each component is small in size. Then, using the quartic relation between size and diameter, \refC{diam} shows that the diameters of these components have order $o(r(n)^{1/4})$ with high probability, proving their negligibility in terms of metric structure. Secondly, Lemmas~\ref{lem:vtxdegree} and~\ref{lem:split} show that these components do not concentrate on a small region, proving their negligibility in terms of measure structure. Then it follows that
\begin{equation}\label{conv1}
\widehat{\bR}_n^+ \coloneqq \left(R_n^+,~\frac{1}{k_n}\cdot d_{R_n^+},~\rho_{\rQ_n},~\frac{1}{|v(R_n^+)|}\cdot \mu_{R_n^+}\right) \convdist \bm_\infty
\end{equation}
for the pointed GHP topology, as shown in \refP{prop:uan}.

On the other hand, we prove that $L_n \coloneqq L(\rQ_n)$ has $\Omega\left(\frac{n}{(\ln n)^2}\right)$ vertices with high probability, shown in \refP{allocationprop} and \refC{Rcor}. Note that conditioned on its size, $L_n$ is a uniform quadrangulation. Then it follows from \citep[Theorem 2]{CLG} that
\begin{equation}\label{conv2}
\widehat{\bL}_n \coloneqq \left(L_n,~\frac{1}{k_n}\cdot d_{L_n},~\rho_{\rQ_n},~\frac{8}{9k_n^4}\cdot \mu_{L_n}\right) \convdist \mathbfcal{P}
\end{equation}
for the local GHP topology. The convergence in \cite{CLG} is stated for the local GH topology, but a slight extension in fact yields the above formulation, as shown in \refP{planeghp}. 

By (\ref{conv1}) and (\ref{conv2}) we easily obtain the joint convergence
\begin{equation}
\label{conv12}
\left(\widehat{\bR}_n^+ ,\widehat{\bL}_n \right) \convdist \left(\bm_\infty,\mathbfcal{P} \right)
\end{equation}
for the local GHP topology, where $\bm_\infty$ and $\mathbfcal{P}$ are independent, as explained in \refL{localgh0}. Finally, we view $\rQ_n$ as a space obtained by gluing $R_n^+$ to $L_n$ at the point $\rho_{\rQ_n}$, and analogously view $\mathbfcal{S}$ as $\bm_\infty$ glued to $\mathbfcal{P}$. \refL{localgh} shows that local GHP convergence is preserved by such a gluing operation. \refT{thm2} then follows easily from (\ref{conv12}).

\subsection{Organization of the Paper}

We associate quadrangulations to a balls-in-boxes model and describe an asymptotically stable distribution for sizes of pendant submaps in \refS{sec:stable}, followed by deriving occupancy in a random allocation model in \refS{sec:allocation}. In \refS{app:number}, we deduce a bound for the number of pendant submaps of the root block, with size and diameter bounds for the submaps derived in \refS{sec:struc}. Then we establish (\ref{conv1}) in \refS{sec:uan}, and complete the proofs of the theorems in \refS{sec:pfthm1}. In \refApp{sec:minibus}, we present definitions for the Brownian plane, with and without minbus (i.e. the Brownian map attachment). Finally, in \refApp{sec:planeghp}, we extend the convergence result of \cite{CLG} to the local GHP topology, following an overview of the scaled Brownian map.

\subsection{Acknowledgement}

I thank Louigi Addario-Berry for suggesting this problem and for advice on improving the proofs. I also thank the referees for careful proofreading and for helpful, minor comments.

\section{Asymptotically Stable Distribution}\label{sec:stable}

Fix $n,r,N\in\N$ with $1 \le N \le n-r$, and let
\begin{equation}\label{eq:Qnrk}
\cQ_{n,r,N} = \left\{\rQ\in\cQ_{n,r}: |\cF(\rQ)| = N \right\},
\end{equation}
where $\cF(\rQ)$ denotes the set of faces enclosed by the facial $2$-cycles in the pre-root-block of $\rQ$. Now, fix $\rQ\in\cQ_{n,r,N}$. Recall that $\left(\rP_i(\rQ):1\le i\le N\right)$ lists the non-singleton submaps lying in $\cF(\rQ)$. When $\rQ$ is random, we are able to recast the behaviour of $\left(|v(\rP_i(\rQ))|:1\le i\le N\right)$ as a balls-in-boxes allocation problem with unlabelled balls (corresponding to the vertices of the submaps) and labelled boxes (corresponding to the faces enclosed by the facial $2$-cycles). With this allocation viewpoint, this section shows that the sizes of these submaps follow an asymptotically stable distribution.

Next, for $m,k\in\N$, write $\cB_{m,k} = \left\{(y_1,\cdots,y_k)\in \N^k: \sum_{i=1}^k y_i = m\right\}$
 for the set of possible allocations of $m$ unlabelled balls in $k$ labelled boxes. For $1\le i\le N$, let
\begin{equation}\label{Ydef}
 Y_i(\rQ)= |v(\rP_i(\rQ))|-1.
\end{equation}
Since $|v(\rP_i(\rQ))|\ge2$, we have $Y_i(\rQ)\ge1$. It is easily seen that $\sum_{i=1}^N Y_{i}(\rQ) = n-r$ and $(Y_i(\rQ):1\le i\le N)\in \cB_{n-r,N}$. We call $(Y_i(\rQ):1\le i\le N)$ the {\em allocation associated with $\rQ$}. Conversely, for a given allocation $(y_i:1\le i\le N)\in \cB_{n-r,N}$, there are multiple rooted quadrangulations $\rQ'\in\cQ_{n,r,N}$ such that $|v(\rP_i(\rQ'))|-1 = y_i$ for each $	1\le i\le N$. We call these the {\em quadrangulations associated with $(y_i:1\le i\le N)\in \cB_{n-r,N}$}. The number of such quadrangulations $\rQ'$ is given in \refL{quadcount}.

Furthermore, we use that there exists $\psi:\N\to\R$ with $\psi(x)\to0$ as $x\to\infty$ and with $\psi(x)>-1$ for all $x\in\N$ such that for all integer $k\ge2$,
\begin{equation}\label{eq:quadsize}
|\cQ_k| = \frac{2}{\sqrt{\pi}} \frac{12^{k-2} (1+\psi(k-1))}{(k-1)^{5/2}};
\end{equation}
 see \cite[Proposition 3.1]{ABW} or \cite{BFSS}. (There are $\frac{2}{\sqrt{\pi}} \frac{12^k (1+o(1))}{k^{5/2}}$ rooted maps with $k$ {\em edges}, so by Tutte's bijection and Euler's formula there are $\frac{2}{\sqrt{\pi}} \frac{12^k (1+o(1))}{k^{5/2}}$ rooted quadrangulations with $k$ {\em faces}, or with $k+2$ {\em vertices}, as $k\to\infty$.) Now, for all $k\in\N$, let
\begin{equation}\label{pdef}
w(k) =\frac{1+\psi(k)}{k^{5/2}},~\overline{w} = \sum_{\ell=1}^\infty w(\ell),~p(k) = \frac{w(k)}{\overline{w}}.
\end{equation}
Since $(p(k): k\in\N)$ is a probability distribution, we may associate it with a random variable $\xi$ such that $\p{\xi = k} = p(k)$ for all $k\in\N$. Let $\xi_{1},\xi_2,\ldots$ be independent copies of $\xi$. For $k\in\N$, write $S_k = \sum_{i=1}^k\xi_i$. 

\begin{prop}\label{cond}
Fix $n,r,N\in\N$ with $1\le N\le n-r$ and $r>1$, and let $(y_i: 1\le i\le N)\in \cB_{n-r,N}$. Then for $\rQ_n\in_u\cQ_{n,r,N}$, 
\[
\p{Y_i(\rQ_n) = y_i, 1\le i\le N}
 = \p{\xi_i = y_i, 1\le i\le N~\big\vert~S_{N}  = n-r}.
 \]
\end{prop}

For any $y=(y_i: 1\le i\le N)\in \cB_{n-r,N}$, let $\Lambda_{n,r,N}^y$ be the number of quadrangulations $\rQ$ in $\cQ_{n,r,N}$ associated with $y$ such that $Y_i(\rQ)=y_i$ for $1\le i\le N$. We start by deriving a formula for $\Lambda_{n,r,N}^y$ in \refL{quadcount} before proving \refP{cond}. For cleanness of presentation, we do not consider the case $r=1$ where the root block is an edge. Recall that $\cR_r$ is the set of rooted $2$-connected quadrangulations of $r$ vertices. 
 
 \begin{lem}\label{quadcount}
Fix $n,r,N\in\N$ with $1\le N\le n-r$ and $r>1$. Then for $y=(y_i: 1\le i\le N)\in \cB_{n-r,N}$, we have $\Lambda_{n,r,N}^y=\left\vert \cR_{r}\right\vert {N+2r-4\choose N} 2^N \prod_{i=1}^{N} \left\vert \cQ_{y_i+1}\right\vert$.
 \end{lem}
 
 \begin{proof}
To build a quadrangulation $\rQ\in\cQ_{n,r,N}$ associated with the allocation $y=(y_i: 1\le i\le N)\in \cB_{n-r,N}$, proceed as follows.
\begin{enumerate}
\item Let $\rR\in \cR_{r}$. Endow each edge of $\rR$ with an orientation so that the tail precedes the head in breadth-first order. List the resulting oriented edges as $({e_{i}}: 1\le i\le |e(\rR)|)$ in the increasing order of $\prec_{\rR}$. Let $e_0$ be a copy of $e_1$ lying to the left of $e_1$, using which we can locate the root edge among multiple edges; see \citep[Proposition 1.7]{ABW}.
\item Choose a vector $(m_i:0\le i\le |e(\rR)|)\in \N^{|e(\rR)|+1}$ with $\sum_{i=0}^{2r-4} m_i=|e(\rR)|+1+N= 2r-4+1+N$. Then for each $0\le i\le |e(\rR)|$, split ${e_i}$ into $m_i$ copies (if $m_i=1$ then there is no split), resulting in $N+1$ facial $2$-cycles. Collapse the $2$-cycle formed by the rightmost copy of $e_0$ and the leftmost copy of $e_1$, and root the map at the resulting edge. List the $N$ faces enclosed by the remaining facial $2$-cycles as $f(1),\ldots,f(N)$, in the canonical order of these faces as described in \refS{sec:graphintro}.
\item For each $1\le i\le N$ let $\rQ_{i}=(Q_i,e_i)\in \cQ_{y_i+1}$, where $e_i$ orients from $u_i$ to $v_i$.
\item For each $1\le i\le N$, choose one of the two resulting corners incident to $f(i)$ and denote it $c(i)$. Attach $\rQ_{i}$ to $c(i)$ by identifying $u_i$ with the vertex of $\rR$ incident to $c(i)$, then add another edge with endpoints $u_i$ and $v_i$, drawn so as to quadrangulate the face $f(i)$.
\end{enumerate}
In step (1), the number of choices for $\rR$ is equal to $|\cR_{r}|$. In step (2), the number of sequences $(m_i\in\N:0\le i\le 2r-4)$ with $\sum_{i=0}^{2r-4} m_i=2r-4+1+N$ is equal to ${N+2r-4\choose N}$. The number of choices in step (3) is $\prod_{i=1}^{N} \left\vert \cQ_{y_i+1}\right\vert$. In step (4), for each $1\le i\le N$, there are two ways to choose $c(i)$, so the total number of choices is $2^{N}$. The proof is then concluded by multiplying the previous four numbers of choices.
\end{proof}

\begin{proof}[{\bf Proof of \refP{cond}}]
Let $y=(y_i: 1\le i\le N)\in \cB_{n-r,N}$. Note that $|\cQ_{y_i+1}|  = \frac{2}{\sqrt{\pi}} \frac{12^{y_i-1}(1+\psi(y_i))}{y_i^{5/2}}$ for $1\le i\le N$ and $\sum_{i=1}^{N}(y_i-1)= n-r-N$. \refL{quadcount} then yields that $\Lambda_{n,r,N}^y=\left\vert \cR_{r}\right\vert {N+2r-4\choose N}  \left( \frac{1}{3\sqrt{\pi}}\right)^{N} 12^{n-r} \prod_{i=1}^{N}  \frac{1+\psi(y_i)}{y_i^{5/2}}$. Furthermore, given any $\rQ\in\cQ_{n,r,N}$ associated with $y$, we have $Y_i(\rQ)=y_i$ for $1\le i\le N$. So for $\rQ_n\in_u\cQ_{n,r,N}$, $\p{Y_i(\rQ_n)=y_i,1\le i\le N} = \frac{\Lambda_{n,r,N}^y}{\sum\limits_{z\in \cB_{n-r,N}}\Lambda_{n,r,N}^{z}}$. It follows that $\p{Y_i(\rQ_n)=y_i,1\le i\le N}=Z \prod_{i=1}^{N}\frac{1+\psi(y_i)}{y_i^{5/2}}$, for some normalizing constant $Z>0$. Finally, recalling the definition of $p(k)$ from (\ref{pdef}), we easily obtain $\p{Y_{i}(\rQ_n) = y_i, 1\le i\le N} = Z  \overline{w}^N\prod_{i=1}^{N} p(y_i)$ and $
\p{\xi_i=y_i,1\le i\le N\big\vert S_{N}= n-r} = \frac{\p{\xi_i = y_i,1\le i\le N}}{\p{S_{N} =n-r}}
= \p{S_{N}= n-r}^{-1} \prod_{i=1}^{N} p(y_i)$. The proposition follows immediately by comparing these two probabilities.
\end{proof}

\section{Random Allocation with Varying Balls-to-Boxes Ratio}\label{sec:allocation}

Recall from \refS{sec:outline} that $r:\N\to\N$ is a function with $r(n)>(\ln n)^{25}$ for all $n$ and $r(n)=o(n)$. In the remainder of the paper, for each $n\in\N$ write $m(n) = n-r(n)$, and let $N:\N\to\N$ be such that $1\le N(n) \le m(n)$. $N(n)$ corresponds to the number of facial $2$-cycles in the pre-root-block of a random quadrangulation with $n$ vertices. Also recall from \refS{sec:stable} that for $k\in\N$, $\p{\xi=k}=p(k)$ where $p(k)$ is given in (\ref{pdef}), and $S_k=\sum_{i=1}^k\xi_i$ where $\xi_1,\xi_2,\ldots$ are independent copies of $\xi$. 

This section aims to describe the law of $(\xi_i:1\le i\le N(n))$ conditioned on $S_{N(n)}=m(n)$. As discussed at the start of last section, this is a random allocation problem, with $m(n)$ unlabelled balls and $N(n)$ labelled boxes in total, viewing $\xi_i$ as the number of balls in the $i$:th box. There are many established results for balls-in-boxes models where the number of balls is proportional to the number of boxes; see the survey by \citet{J}. However, here we need to allow the balls-to-boxes ratio $\frac{m(n)}{N(n)}$ to tend to infinity, so a variant of the established work is needed. We accomplish this in Propositions~\ref{allocationprop} and~\ref{allocationprop2}, extending the result of \citep[Theorem 19.34]{J}. These bounds can be applied to a uniform quadrangulation in $\cQ_{n,r(n),N(n)}$, by using \refP{cond}. Some analysis in this section is related to a so-called one-jump principle for random walks; see, for example, \citep[Section 2.3]{AS} in a different, easier context.

Given $k\in\N$, for $(x_1,\ldots,x_k)\subset \R^k$, write $(x_{(1)},\ldots,x_{(k)})=(x_{k,(1)},\ldots,x_{k,(k)})$ as its decreasing ordered sequence. In particular, write $\xi_{(1)} = \xi_{N(n),(1)} = \max(\xi_i: 1\le i\le N(n))$. 

Let $\nu=\E{\xi}$; clearly, $\nu<\infty$.

 \begin{prop}\label{allocationprop}
Given $\limsup\limits_{n\to\infty} \frac{\nu N(n)}{m(n)}<1$, $\p{ \xi_{(1)} \le  \frac{m(n)}{(\ln n)^2} ~\Big\vert~S_{N(n)}=m(n)} = O\left(n^{-10}\right)$.
\end{prop}

\begin{prop}\label{allocationprop2}
Given $\limsup\limits_{n\to\infty} \frac{\nu N(n)}{m(n)}<1$,
$\p{ \xi_{(1)} >\frac{m(n)}{(\ln n)^2},\xi_{(2)}> r(n)^{5/6}~\Big\vert~S_{N(n)}=m(n)} = O\left( N(n) \left(\ln  n\right)^{5} r(n)^{-5/4} \right)$.
 \end{prop}

Before proving Propositions~\ref{allocationprop} and~\ref{allocationprop2}, we state an immediate application to quadrangulations, below. Recall the definitions of $\cQ_{n,r,k}$ from (\ref{eq:Qnrk}) and $Y_i(\cdot)$ from (\ref{Ydef}). 

 \begin{cor}\label{quadregime1}
Suppose that $\limsup_{n\to\infty} \frac{\nu N(n)}{m(n)}<1$. Then for $\rQ_n\in_u\cQ_{n,r(n),N(n)}$,
\begin{equation}\label{eq:lar}
\p{ Y_{(1)}(\rQ_n)\le \frac{m(n)}{(\ln  n)^2}} = O\left(n^{-10}\right),
\end{equation}
and $\p{ Y_{(1)}(\rQ_n)>\frac{m(n)}{(\ln n)^2},Y_{(2)}(\rQ_n)> r(n)^{5/6}} =  O\left( N(n) \left(\ln n\right)^{5} r(n)^{-5/4} \right);$ it follows that
\begin{equation}\label{eq:sec}
\p{Y_{(2)}(\rQ_n)> r(n)^{5/6}}= O\left( N(n) \left(\ln n\right)^{5} r(n)^{-5/4} \right).
\end{equation}
 \end{cor}

The first two equalities follow immediately from Propositions \ref{cond}, \ref{allocationprop}, and \ref{allocationprop2}. For the last assertion, simply note that for $\rQ_n\in_u\cQ_{n,r(n),N(n)}$,
\[\left\{Y_{(2)}(\rQ_n)> r(n)^{5/6}\right\}
\subset \left\{ Y_{(1)}(\rQ_n) >\frac{m(n)}{(\ln n)^2},Y_{(2)}(\rQ_n)> r(n)^{5/6}\right\}\bigcup\left\{Y_{(1)}(\rQ_n)\le \frac{m(n)}{(\ln n)^2}\right\}.\]

\refC{quadregime1} gives size-bounds for the largest and second largest submaps pendant to the root block of a uniform quadrangulation in $\cQ_{n,r(n),N(n)}$. In the next section, we deduce a bound on the number of facial $2$-cycles in the pre-root-block of a uniform quadrangulation in $\cQ_{n,r(n)}$, which entails us to apply the bounds of \refC{quadregime1} to the latter setting.

Now we turn to establishing Propositions~\ref{allocationprop} and~\ref{allocationprop2}, starting with two lemmas related to sums of asymptotically stable distributions. 

For $k,\ell\in\N$, write $\xi_\ell^k = \xi_\ell \I{\xi_\ell\le k}$ and $S_\ell^k = \sum_{i=1}^\ell \xi_i^k$. 

\begin{lem}\label{lem:chernoff}
For $m\in\N$ and $x>0$, $\p{S_{m}^{k}\ge x}\le  e^{-\frac{x}{k}+ \frac{\nu m}{k}(1+o(1))}$ as $k\to\infty$.
\end{lem}

\begin{proof}
Since $(\xi_i^k:i\in\N)$ are iid, for $x>0$ and $s>0$, by Chernorff inequality,
\begin{align}
\p{S_{m}^{k}\ge x}\le e^{-sx} \E{e^{s S_m^{k}}} = e^{-sx} \left(\E{e^{s \xi_1^{k}}}\right)^{m},\label{eq:upbd}
\end{align}
where $\E{e^{s \xi_1^{k}}}= \p{\xi>k}+ \sum_{t=1}^{k} \p{\xi=t}  e^{st}
\le 1+ s \nu + \sum_{t=1}^{k} \p{\xi=t} \left(e^{st}-1 - st\right)$. Furthermore, (\ref{pdef}) yields that there exists $c>0$, not depending on $k$, such that for all $t\ge1$, $\p{\xi=t}\le c t^{-5/2}$. It is easily seen that, if $st\le 1$, we have $e^{st}-1 - st\le  s^2 t^2$, and so $\E{e^{s\xi_1^{k}}}\le1+s\nu+ c\sum_{t=1}^{k} t^{-5/2} s^2 t^2$. Now take $s= \frac1k$, then $\sum_{t=1}^{k} t^{-5/2} s^2 t^2
= \sum_{t=1}^{k} \frac{t^{-1/2}}{k^2}
=O\left(  \frac{k^{1/2}}{k^2}\right)
=o\left(\frac{1}{k}\right)$ as $k\to\infty$. Altogether, $\E{e^{\frac{\xi_1^{k}}{k}}}\le 1+\frac{\nu}{k} + o\left(\frac{1}{k}\right) \le e^{\frac{\nu}{k}(1+o(1))}$. With $s=\frac1k$, the lemma follows immediately from (\ref{eq:upbd}).
\end{proof}

Recall that $S_k =\sum_{i=1}^k \xi_i$ for $k\in\N$ and $\nu=\E{\xi}$.

\begin{lem}\label{lem:sum}
Fix $\lambda\in (0,1)$. There exists $\delta=\delta(\lambda)>0$ such that for sufficiently large integers $N$ and $m$ with $\lambda m\ge \nu N = \E{S_N}$, we have $\p{S_N=m}\ge  \frac{\delta N }{m^{5/2}}$.
\end{lem}

\begin{proof}
Fix large enough $N,m\in\N$ with $\lambda m\ge\nu N$ and $N^{2/3}<\frac{(1-\lambda)m}{2}$. For $1\le i\le N$, let $
E_i = \left\{ \left\vert m -\xi_i - \nu N\right\vert < N^{2/3},~\max\limits_{j=1,\ldots,i-1,i+1,\ldots, N} \xi_j \le \frac{(1-\lambda)m }{2}, ~S_{N}=m\right\}$. Since $\nu N\le  \lambda m$, if the event $E_i$ occurs, then $
\xi_i > m - \nu N - N^{2/3} \ge 
(1-\lambda) m - N^{2/3} > \frac{(1-\lambda)m}{2}\ge \max\limits_{j=1,\ldots,i-1,i+1,\ldots,N}\xi_j$. So the events $E_1,\ldots,E_{N}$ are disjoint. It follows by symmetry and independence that $\p{S_{N}=m}
\ge N \p{E_{N}}$, which is lower bounded by
\begin{align*}
&N \p{|m-\xi_N - \nu N| < N^{2/3},S_N=m}- N^2 \p{\xi_N > \frac{(1-\lambda)m}{2}, \xi_{N-1}>\frac{(1-\lambda)m}{2},S_N=m}\\
&\ge N \sum_{k=\lceil \nu N - N^{2/3}\rceil}^{\lfloor \nu N+N^{2/3}\rfloor} \p{S_{N-1}=k} \p{\xi=m-k} - N^2
\p{\xi> \frac{(1-\lambda)m}{2}} \sup_{\lceil \frac{(1-\lambda)m}{2}\rceil\le \ell \le m} \p{\xi=\ell}.
\end{align*}
So with $c = \inf \left\{\ell^{5/2}\p{\xi=\ell}: \ell \in\N\right\}>0$, for all $k$ in the above sum, $\p{\xi=m-k}\ge \frac{c}{m^{5/2}}$. Similarly, with $
d = \sup\left\{\ell^{5/2} \p{\xi = \ell}: \ell \in \N  \right\}<\infty$, $\p{\xi>\frac{(1-\lambda)m}{2}} \le \frac{2d} {3}\left(\frac{2}{(1-\lambda)m}\right)^{3/2}$, and for $\ell\ge \frac{(1-\lambda)m}{2}$ we have $\p{\xi = \ell}\le \frac{d}{ \ell^{5/2}}\le d \left(\frac{2}{(1-\lambda)m} \right)^{5/2}$. Altogether,
\begin{align*}
\p{S_{N}=m}\ge\frac{Nc }{m^{5/2} }\sum_{k=\lceil \nu N- N^{2/3}\rceil}^{\lfloor \nu N +N^{2/3}\rfloor} \p{S_{N-1}=k} - N^2  \frac{2d} {3}\left(\frac{2}{(1-\lambda)m}\right)^{3/2} d \left(\frac{2}{(1-\lambda)m} \right)^{5/2}.
\end{align*}
Since $\xi$ is in the domain of attraction of a $\frac{3}{2}$-stable random variable, the fluctuation of $S_{N-1}$ around its mean is of order $N^{2/3}$. By decreasing $c$ if necessary, we may assume that $\p{\left\vert S_{N-1} -\nu N\right\vert < N^{2/3}}\ge c$ and obtain that $\p{S_{N}=m}\ge \frac{c^2N}{m^{5/2}} - \frac{2d^2}{3}\left( \frac{2}{1-\lambda}\right)^4\frac{N^2}{m^4}\ge \frac{c^2N}{2m^{5/2}}$, the last inequality holding since for any $\veps>0$, we have $\frac{N^2}{m^4}<\frac{\veps N}{m^{5/2}}$ for large enough $N$ and for all $m$ permitted by the lemma.
\end{proof}

 \begin{proof}[{\bf Proof of \refP{allocationprop}}]
First, fix $k\in\N$, and recall that $S_{N(n)}^k = \sum_{i=1}^{N(n)}\xi_i^k$ where $\xi^k_i =\xi_i \I{\xi_i\le k}$. Considering which summand of $S_{N(n)}$ is largest leads to $\{\xi_{(1)}=k\}\cap \{S_{N(n)}=m(n) \} \subset \bigcup_{i=1}^{N(n)} \{\xi_i=k\}\cap \{S_{N(n)}^k - \xi_i^k=m(n)-k\}$. Recalling $\overline{w} = \sum_{\ell=1}^\infty w(\ell)$, by symmetry and independence it is easily seen that
\[
\p{\xi_{(1)}=k,S_{N(n)}=m(n)}\le N(n) \frac{1+\psi(k)}{k^{5/2}\overline{w}} \p{S_{N(n)-1}^k=m(n)-k}.
\]
Furthermore, since $\limsup_{n\to\infty}\frac{\nu N(n)}{m(n)}<1$, there exists $\lambda\in(0,1)$ with $\limsup_{n\to\infty}\frac{\nu N(n)}{m(n)}\le \lambda$. Then by \refL{lem:sum}, there exists $\delta=\delta(\lambda)>0$ such that for all $n$ large enough we have $\p{S_{N(n)}=m(n)}\ge \frac{\delta N(n)}{m(n)^{5/2}}$, so Bayes formula now gives
\[
\p{\xi_{(1)}=k~\big\vert~ S_{N(n)}=m(n)}\le  \frac{1+\psi(k)}{\delta \overline{w}} \left(\frac{m(n)}{k}\right)^{5/2}\p{S_{N(n)-1}^k=m(n)-k}.
\]
Note that if $S_{N(n)}=m(n)$ then $\xi_{(1)}\ge m(n)/N(n)$. So with $N(n)< (\ln n)^2$ we have $\p{\xi_{(1)}\le \frac{m(n)}{(\ln n)^2}~\Big\vert~S_{N(n)}=m(n)}=0$. It thus remains to consider the case $N(n)\ge (\ln n)^2$:
\begin{align}
&\p{\xi_{(1)}\le \frac{m(n)}{(\ln n)^2}~\Big\vert~S_{N(n)}=m(n)}
\le \sum_{k=\lfloor \frac{m(n)}{N(n)}\rfloor}^{\lfloor\frac{m(n)}{(\ln n)^2}\rfloor} \frac{1+\psi(k)}{\delta \overline{w}} \left(\frac{m(n)}{k}\right)^{5/2} \p{S_{N(n)-1}^k=m(n)-k}\notag\\
&\le \frac{N(n)^{5/2}}{\delta\overline{w}}\sup_{\lfloor\frac{m(n)}{N(n)}\rfloor\le k\le \lfloor\frac{m(n)}{(\ln n)^2}\rfloor} (1+\psi(k)) \p{S_{N(n)-1}^{\lfloor\frac{m(n)}{(\ln n)^2}\rfloor}\ge m(n)\left(1-\frac{1}{(\ln n)^2}\right)}.\label{in:sum}
\end{align}
Since $m(n)=n(1+o(1))$, we have $\frac{m(n)}{(\ln n)^2}\to\infty$. Therefore, by \refL{lem:chernoff}, 
\begin{align*}
&\p{S_{N(n)-1}^{\lfloor\frac{m(n)}{(\ln n)^2}\rfloor}\ge m(n) \left(1-\frac{1}{(\ln n)^2}\right)} 
\le \exp\left(- \frac{m(n)}{\lfloor \frac{m(n)}{(\ln n)^2}\rfloor} \left(1-\frac{1}{(\ln n)^2}\right) + \frac{\nu N(n)}{\lfloor\frac{m(n)}{(\ln n)^2}\rfloor}(1+o(1))\right).
\end{align*}
Since $\limsup_{n\to\infty}\frac{\nu N(n)}{m(n)}<1$, there is $\veps>0$ so that for large enough $n$, $- \frac{m(n)}{\lfloor\frac{m(n)}{(\ln n)^2}\rfloor} \left(1-\frac{1}{(\ln n)^2}\right) + \frac{\nu N(n)}{\lfloor\frac{m(n)}{(\ln n)^2}\rfloor}(1+o(1)) 
<
-\veps(\ln n)^2$. Hence, $\p{S_{N(n)-1}^{\lfloor\frac{m(n)}{(\ln n)^2}\rfloor}\ge m(n) \left(1-\frac{1}{(\ln n)^2}\right)}
< e^{-\veps(\ln n)^2}$. Combined with (\ref{in:sum}), $
 \p{\xi_{(1)}\le \frac{m(n)}{(\ln n)^2}~\Big\vert~S_{N(n)}=m(n)}=O\left( \frac{N(n)^{5/2}}{n^{\veps\ln n}}\right) = O\left(n^{-10}\right)$.
\end{proof}

 \begin{proof}[{\bf Proof of \refP{allocationprop2}}]
By symmetry and independence,
\begin{align*}
&\p{\xi_{(1)}> \frac{ m(n)}{(\ln n)^2},\xi_{(2)}>  r(n)^{5/6}~\Big\vert~ S_{N(n)}=m(n)} \\
\le&~N(n)^2 \p{\xi_{N(n)}>\frac{ m(n)}{(\ln n)^2},\xi_{N(n)-1}> r(n)^{5/6}~\Big\vert~ S_{N(n)}=m(n)}\\
= &~ N(n)^2 \sum_{i=\lfloor \frac{ m(n)}{(\ln n)^2}+1\rfloor}^{m(n)}\frac{\p{\xi_{N(n)}=i}\p{\xi_{N(n)-1}>r(n)^{5/6},
S_{N(n)-1}=m(n)-i}}{\p{S_{N(n)}=m(n)}}.
\end{align*}
Note that for $i\ge \frac{ m(n)}{(\ln n)^2}$, $\p{\xi=i}  = O\left(\left(\frac{ m(n)}{(\ln n)^2}\right)^{-5/2}\right)$. Furthermore, since $\xi$ is in the domain of attraction of a $\frac32$-stable random variable, we have $\p{\xi>r(n)^{5/6}} = O\left(r(n)^{-\frac32 \cdot \frac56}\right)=O\left(r(n)^{-5/4}\right)$. Together with \refL{lem:sum}, 
 \begin{align}
&\p{\xi_{(1)}> \frac{ m(n)}{(\ln n)^2},\xi_{(2)}>  r(n)^{5/6}~\Big\vert~ S_{N(n)}=m(n)}
= \frac{O\left( N(n)^2 \left(\frac{ m(n)}{(\ln n)^2}\right)^{-5/2} r(n)^{-5/4}\right)} {\p{S_{N(n)}=m(n)}}\notag\\
&= O\left( N(n)^2 \left(\frac{ m(n)}{(\ln n)^2}\right)^{-5/2} r(n)^{-5/4} N(n)^{-1}m(n)^{5/2}\right)
= O\left( N(n) \left(\ln n\right)^{5} r(n)^{-5/4} \right)\notag.\qedhere
\end{align}
\end{proof}

\section{The Number of Facial $2$-Cycles in the Pre-Root-Block}\label{app:number}

 This section shows that for $\rQ_n\in_u\cQ_{n,r}$ with appropriate $r$, we have $|\cF(\rQ_n)|< 3r$ with high probability, recalling that $|\cF(\cQ_n)|$ is the number of facial $2$-cycles in the pre-root-block of $\rQ_n$. Together with the assumptions that $r(n)>(\ln n)^{25}$ and $r(n)=o(n)$, this verifies that the conditions in \refC{quadregime1} hold with high probability, paving the way to proving condensation phenomena for $\rQ_n$ in \refS{sec:struc}. Recall that $\nu = \E{\xi}$.

\begin{prop}\label{prop:Nsize}
Fix $\lambda\in(0,1)$. There exists $c=c(\lambda)>0$ such that the following holds. For sufficiently large integers $r$ and $n$ with $ (n-r)\lambda\ge 2\nu r$, given $\rQ_n\in_u\cQ_{n,r}$, we have
\[
\p{|\cF(\rQ_n)|\ge 3 r} \le  c n^{5/2} \left(\frac{4}{9}\right)^{r}.
\]
\end{prop}

We start by deriving a lemma about $|\cQ_{n,r,k}|$, recalling the definition of $\cQ_{n,r,k}$ from (\ref{eq:Qnrk}).

\begin{lem}\label{cor:2}
Fix $n,r,k\in\N$ with $1<r< n$ and $2r-4\le k\le n-r$. Then
\[
\left\vert \cQ_{n,r,k}\right\vert\le \left\vert \cQ_{n,r,2r-4}\right\vert\left(\frac{4}{9}\right)^{k-2r+4} \frac{\p{S_k=n-r}}{\p{S_{2r-4}=n-r}}.
\]
\end{lem}

\begin{proof}
First, let $M(z)$ be the generating function of rooted quadrangulations with $z$ marking the number of faces (i.e. the number of vertices minus two). That is, $M(z) = \sum_{\ell=1}^\infty |\cQ_{\ell+2}| z^\ell$.
Note that $M\left(\frac{1}{12}\right) = \frac13$ by \citep[Proposition 4]{BFSS}. Furthermore, by (\ref{eq:quadsize}) we have $|\cQ_{\ell+2}| = \frac{2}{\sqrt{\pi}} \frac{12^{\ell} (1+\psi(\ell+1))}{(\ell+1)^{5/2}}$ for $\ell\in\N$, and we take $|\cQ_2|=1$ since we view a single edge as a quadrangulation. Thus, $\sum_{\ell=1}^\infty \frac{2}{\sqrt{\pi}} \frac{1+\psi(\ell)}{\ell^{5/2}} = \sum_{\ell=0}^\infty |\cQ_{\ell+2}|\frac{1}{12^\ell} = 1+M\left(\frac{1}{12}\right) = \frac43$, and so $\overline{w}\coloneqq \sum_{\ell=1}^\infty \frac{1+\psi(\ell)}{\ell^{5/2}}= \frac{2\sqrt{\pi}}{3}$. This yields that, for $i\in\N$, $\p{\xi = i}= \frac{\frac{1+\psi(i)}{i^{5/2}}}{\overline{w}} = \frac{3}{2\sqrt{\pi}} \frac{1+\psi(i)}{i^{5/2}}$, recalling the distribution of $\xi$ from (\ref{pdef}). Together with (\ref{eq:quadsize}), it follows that $\sum\limits_{y_1+\ldots+y_k=n-r} \prod_{i=1}^k |\cQ_{y_i+1}|
= \left(\frac{2}{\sqrt{\pi}}\right)^k 12^{n-r-k} \left(\frac{2\sqrt{\pi}}{3}\right)^k \p{S_k = n-r}$, where $S_k = \sum_{i=1}^k \xi_i$ and $\xi_1,\ldots,\xi_k$ are iid. Moreover, by summing over sequences of $(y_1,\ldots,y_k)\in \cB_{n-r,k}$ with $y_1+\ldots +y_k = n-r$, \refL{quadcount} gives $\left\vert \cQ_{n,r,k}\right\vert = |\cR_r|  {2r-4+k\choose k} 2^k  \sum\limits_{y_1+\ldots+y_k=n-r} \prod_{i=1}^k |\cQ_{y_i+1}|$. Combining these equalities, we obtain $
|\cQ_{n,r,k}| = |\cR_r|{2r-4+k\choose k} \p{S_k=n-r}\left(\frac{2}{9}\right)^k 12^{n-r}$. Finally, note that for $a\in\N$, ${a+k+1\choose k+1} = \frac{a+k+1}{k+1}{a+k\choose k}$. Since $\frac{a+k}{k}$ decreases in $k$, for $k\ge 2r-4$, we easily get $\frac{\left\vert \cQ_{n,r,k}\right\vert}{\left\vert \cQ_{n,r,2r-4}\right\vert } 
\le \left(\frac{2r-4+2r-4}{2r-4} \frac{2}{9} \right)^{k-2r+4}  \frac{\p{S_k=n-r}}{\p{S_{2r-4}=n-r}}$.
\end{proof}

\begin{proof}[{\bf Proof of \refP{prop:Nsize}}]
Fix large enough integers $n$ and $r$ with $(n-r)\lambda\ge 2\nu r$. By \refL{lem:sum}, there exists $\delta=\delta(\lambda)>0$ with $\p{S_{2r-4}=n-r} \ge  \frac{(2r-4)\delta}{ n^{5/2}}$. The bound in \refL{cor:2} then gives that, for $2r-4\le k\le n-r$, $\left\vert \cQ_{n,r,k}\right\vert \le 
\left\vert \cQ_{n,r,2r-4}\right\vert  \left(\frac{4}{9}\right)^{k-2r+4}\frac{n^{5/2}}{(2r-4)\delta}$. It follows that, for $\rQ_n\in_u \cQ_{n,r}$, $
\p{|\cF(\rQ_n)| \ge 3r}
\le \sum_{3r\le k\le n-r} \frac{\left\vert \cQ_{n,r,k}\right\vert }{\left\vert \cQ_{n,r,2r-4}\right\vert}
\le \frac{9}{5\delta}\left(\frac49\right)^r n^{5/2}$.
\end{proof}

\section{Condensation in Quadrangulation Conditioned on Root Block Size}\label{sec:struc}

In this section, we show that for $\rQ_n\in_u\cQ_{n,r(n)}$, with high probability, there is condensation in $\rQ_n$ (see \cite{A} for an overview of condensation in random maps), and the root block $R(\rQ_n)$ does not separate two large submaps; see Corollaries \ref{Rcor} and \ref{Rcor2}. The corollaries are immediate consequences of \refP{prop:Nsize}, (\ref{eq:lar}), and (\ref{eq:sec}). Their proofs are similar, so we only present the latter one.

Recall from (\ref{Ydef}) that $(Y_i(\rQ_n)+1:1\le i\le |\cF(\rQ_n)|)$ are the sizes of submaps pendant to $R(\rQ_n)$, and write them as $Y_{(1)}(\rQ_n),Y_{(2)}(\rQ_n),\ldots,Y_{(|\cF(\rQ_n)|)}(\rQ_n)$ in decreasing order.

\begin{cor}\label{Rcor}
For $\rQ_n\in_u \cQ_{n,r(n)}$, $\p{ Y_{(1)}(\rQ_n)\le \frac{m(n)}{(\ln n)^2}} = o(1)$.
\end{cor}

\begin{cor}\label{Rcor2}
For $\rQ_n\in_u \cQ_{n,r(n)}$, $\p{Y_{(2)}(\rQ_n)>  r(n)^{5/6}} =  o(1)$.
\end{cor}

\begin{proof}
Fix $\rQ_n\in_u\cQ_{n,r(n)}$. We have
\[
\p{Y_{(2)}(\rQ_n)> r(n)^{5/6}} 
\le
 \p{Y_{(2)}(\rQ_n)>r(n)^{5/6}~\Big\vert~|\cF(\rQ_n)|<3r(n)} + \p{|\cF(\rQ_n)|\ge 3 r(n)}.
\]
Recalling that $r(n)>(\ln n)^{25}$ for all $n$ and $r(n) =o(n)$, \refP{prop:Nsize} yields that $\p{|\cF(\rQ_n)|\ge 3r(n)}=o(1)$. Furthermore, it follows from (\ref{eq:sec}) that
\[
 \p{Y_{(2)}(\rQ_n)>r(n)^{5/6}~\Big\vert~|\cF(\rQ_n)|<3r(n)}
= O\left(r(n)^{-1/4}(\ln n)^{5} \right)=o(1).\qedhere
\]
\end{proof}

Next, for $\rQ_n\in_u\cQ_{n,r(n)}$, we derive a tail bound for the maximal diameter of non-largest submap pendant to the root block. The derivation relies on \cite[Proposition 4]{CS}.

\begin{prop}{\em (\citet{CS}).}\label{diamlem}
Fix $n\in\N$ and let $\rQ_n\in_u\cQ_n$. There exist positive constants $x_0$, $c_1$, and $c_2$ such that for all $x\ge x_0$, $
\p{\diam\left(\rQ_n\right) > x n^{1/4}}\le c_1 e^{-c_2 x}$.
\end{prop}

Combining Corollary \ref{Rcor2} and \refP{diamlem}, we obtain the following corollary easily.

Given $\rQ\in\cQ$, recall that $(\rP_i(\rQ):1\le i\le |\cF(\rQ)|)$ is the sequence of submaps pendant to $R(\rQ)$. Write them in the decreasing order of size as $(\rP_{(i)}(\rQ):1\le i\le |\cF(\rQ)|)$.

\begin{cor}\label{diam}
For $\rQ_n\in_u \cQ_{n,r(n)},x>0$, $\p{\max\limits_{i\ge 2}\diam(\rP_{(i)}(\rQ_n)) \ge  x r(n)^{1/4} }=o(1)$.
\end{cor}

\begin{proof}
For $|\cF(\rQ_n)|<i\le n-r(n)$, write $\rP_{(i)}(\rQ_n)=\emptyset$. Let $k\in\N_{\ge0}$ with $k\le n-r+1$. For any $1\le i\le n-r(n)$, given that $|v(\rP_{(i)}(\rQ_n))|=k$, $\rP_{(i)}(\rQ_n)$ is uniformly distributed over $\cQ_k$ (denoting $\cQ_0=\emptyset$). By \refP{diamlem}, there exist positive constants $x_0$, $c_1$, and $c_2$ such that for all $x\ge x_0$, and for all $1\le i\le n-r(n)$,
\begin{equation}\label{conddiambound}
 \p{\diam(\rP_{(i)}(\rQ_n)) \ge xk^{1/4}~\Big\vert~ |v(\rP_{(i)}(\rQ_n))| =k }\le  c_1 e^{-c_2 x }.
\end{equation}
Now, fix $x>0$. We have
\begin{align*}
&\p{\max\limits_{2\le i\le n-r(n)}\diam(\rP_{(i)}(\rQ_n)) \ge x r(n)^{1/4}}\\
\le&~\p{\max\limits_{2\le i\le n-r(n)}\diam(\rP_{(i)}(\rQ_n)) \ge xr(n)^{1/4},Y_{(2)}(\rQ_n)\le r(n)^{5/6}} + \p{Y_{(2)}(\rQ_n)> r(n)^{5/6}}.
\end{align*}
By \refC{Rcor2}, $\p{Y_{(2)}(\rQ_n)> r(n)^{5/6}}=o(1)$. Next, by a union bound, it follows that
\begin{align}
&\p{\max\limits_{2\le i\le n-r(n)}\diam(\rP_{(i)}(\rQ_n)) \ge x r(n)^{1/4},Y_{(2)}(\rQ_n)\le r(n)^{5/6}}\notag\\
\le&~ \sum_{i=2}^{n-r(n)} \sup_{1\le k\le \lfloor r(n)^{5/6}\rfloor}  \p{\diam\left(\rP_{(i)}(\rQ_n)\right) \ge x r(n)^{1/4}~ \Big\vert ~|v(\rP_{(i)}(\rQ_n))| =k }.\label{doublesum}
\end{align}
For sufficiently large $n$ and for $1\le k\le r(n)^{5/6}$, we have $xr(n)^{1/4} k^{-1/4} \ge xr(n)^{1/24}\ge x_0$, then (\ref{conddiambound}) yields that, for $1\le i\le n-r(n)$,
\begin{align*}
\p{\diam\left(\rP_{(i)}(\rQ_n)\right) \ge x r(n)^{1/4} ~\Big\vert~ |v(\rP_{(i)}(\rQ_n))| =k }
\le c_1 \exp\left(-c_2 x r(n)^{1/24}\right).
\end{align*}
(\ref{doublesum}) then leads to
\begin{align*}
\p{\max\limits_{2\le i\le n-r(n)}\diam(\rP_{(i)}(\rQ_n)) \ge x r(n)^{1/4},Y_{(2)}(\rQ_n)\le r(n)^{5/6}}
\le c_1\exp\left(\ln n  - c_2x r(n)^{1/24} \right).
\end{align*}
By the assumption that $r(n)>(\ln n)^{25}$, we have
\begin{equation}\label{eq:assum}
\ln n  - c_2x r(n)^{1/24} \to -\infty,
\end{equation}
 completing the proof.
\end{proof}

\section{Uniformly Asymptotically Negligible Attachments}\label{sec:uan}

Fix $\rQ_n\in_u\cQ_n$. Recall from \refS{sec:graphintro} that $R^+(\rQ_n) = \rQ_n-\left(v\left(L(\rQ_n)\right)\setminus\{\rho_{\rQ_n}\}\right)$. In what follows, Propositions~\ref{prop:uan1} and~\ref{prop:uan} show that $R^+(\rQ_n)$ and $R(\rQ_n)$ have the same scaling limit, when respectively endowed with the measures in Theorem~\ref{thm1} and~\ref{thm2}. With the measure of Theorem~\ref{thm1}, no mass is assigned to the components of $R^+(\rQ_n)- v(R(\rQ_n))$, so the convergence in \refP{prop:uan1} is easier to establish than that in \refP{prop:uan}.

Recall the definition of pointed metric measure space from \refS{sec:main}. The {\em pointed GHP distance} between pointed metric measure spaces $\bV=(V,d,v,\mu)$ and $\bW=(W,d',w,\mu')$ is $
\dghp^\star\left(\bV,\bW \right) = \inf\left[ \max\left\{ \delta_\rH(\phi(V),\phi'(W)) , \delta_\rP(\phi_*\mu,\phi'_* \mu') , \delta(\phi(v),\phi'(w)) \right\}\right]$ where the infimum is taken over all metric space $(Z,\delta)$ and isometries $\phi,\phi'$ from $(V,d)$ and $(W,d')$ to $(Z,\delta)$, and $\delta_\rH,\delta_\rP$ denote Hausdorff and Prokhorov distances.

Given a graph $G$, we often write $d_G$ for the graph distance on any induced subgraph of $G$, and for the intrinsic metric in its approximating boundedly compact length space. Recall that 
$k_n = \left(\frac{40\cdot r(n)}{21}\right)^{1/4}$ for $n\in\N$.

\begin{prop}\label{prop:uan1}
For $\rQ_n\in_u\cQ_{n,r(n)}$, $\left( R^+(\rQ_n), ~ \frac{1}{k_n}\cdot d_{\rQ_n},~ \rho_{\rQ_n},~\frac{1}{r(n)}\cdot \mu_{R(\rQ_n)}\right) \convdist \bm_\infty$ for the pointed GHP topology.
\end{prop}

\begin{proof}
For $n\in\N$, let $\rQ_n\in_u\cQ_{n,r(n)}$. Write $\widehat{\bR}_n = \left(R(\rQ_n),~\frac{1}{k_n}\cdot d_{\rQ_n},~\rho_{\rQ_n},~\frac{1}{r(n)}\cdot \mu_{R(\rQ_n)} \right)$ and $\widehat{\bR}_n^+ = \left(R^+(\rQ_n),~\frac{1}{k_n}\cdot d_{\rQ_n},~\rho_{\rQ_n},~\frac{1}{r(n)}\cdot \mu_{R(\rQ_n)}\right)$. As discussed in \refS{sec:outline}, $\widehat{\bR}_n\convdist \bm_\infty$ for the pointed GHP topology, so it suffices to show that $\dghp^\star\left(\widehat{\bR}_n,\widehat{\bR}_n^+\right)\convp0$. Write $\cO_n$ for the set of components of $R^+(\rQ_n) - v(R(\rQ_n))$. Since $\widehat{\bR}_n$ and $\widehat{\bR}_n^+$ are equipped with the same measure, it follows from the definition of $\dghp^\star$ that $\dghp^\star\left(\widehat{\bR}_n,\widehat{\bR}_n^+\right)
 \le \frac{1}{k_n} \max_{G\in \cO_n} (\diam(G)+1)$. By \refC{diam}, $ \max_{G\in \cO_n} \diam(G) = o(r(n)^{1/4})$ with $1-o(1)$ probability. Since $k_n = \Theta(r(n)^{1/4})$, it follows that $\dghp^\star\left(\widehat{\bR}_n,\widehat{\bR}_n^+\right) \convp0$.
\end{proof}

\begin{prop}\label{prop:uan}
Given $\rQ_n\in_u\cQ_{n,r(n)}$, for the pointed GHP topology,
\begin{equation}\label{conv:uan}
\left( R^+(\rQ_n), ~\frac{1}{ k_n}\cdot d_{\rQ_n},~ \rho_{\rQ_n},~\frac{1}{|v(R^+(\rQ_n))|}\cdot \mu_{R^+(\rQ_n)}\right) \convdist \bm_\infty.
\end{equation}
 \end{prop}

We devote the rest of this section to proving \refP{prop:uan}. We proceed in the following two steps to show that given $\rQ_n\in_u\cQ_{n,r(n)}$, with high probability there are no $w\in v(R(\rQ_n))$ to which an overly large mass of pendant submaps attach. First, in \refL{lem:vtxdegree} we prove a tail bound for the maximum degree in a uniform rooted $2$-connected quadrangulation. Secondly, we prove that with high probability no edge of $R(\rQ_n)$ is subdivided many times in $\rQ_n$, as shown in \refL{lem:split}.

Given a graph $G$, for $u\in v(G)$, write $\deg_G(u)$ for the degree of $u$ in $G$. Recall that $\cR_r$ denotes the set of rooted $2$-connected quadrangulations with $r$ vertices.

\begin{lem}\label{lem:vtxdegree}
Fix $x\in\N$. For any $\veps>1/2$ there exists $B>0$ such that for all $r\in\N$, given $\rR_{r}\in_u\cR_{r}$, $\p{\max_{u\in v(\rR_r)} \deg_{\rR_{r}} (u) = x}\le B \veps^x r^{5/3}$.
\end{lem}

\begin{proof}
This straightforward proof is a slight modification of the proof for \citep[Lemma 7.2]{ABW}. First, for any $\veps>1/2$ there exists $c>0$ such that for all $n\in\N$, given $\rQ_n\in_u\cQ_n$,
\begin{equation}\label{in:degbd}
\p{\max_{v\in v(\rQ_n)}\deg_{\rQ_n}(v) = x}\le  c\veps^xn;
\end{equation}
see \citep[Theorem 2.1 (a)]{BC}. Now, fix $r\in\N$ and write $n=\lfloor 15r/7\rfloor$. Let $\rR_r\in_u\cR_{r}$ and $\rQ_n\in_u\cQ_n$. Next, let $\rR'(\rQ_n)$ be the largest block of $\rQ_n$, rooted at its $\prec_{\rQ_n}$-minimal edge; if there are multiple blocks of size $|v(\rR'(\rQ_n))|$, then among these blocks we choose $\rR'(\rQ_n)$ to be the one whose root edge is $\prec_{\rQ_n}$-minimal. Given that $|v(\rR'(\rQ_n))|=r$, $\rR'(\rQ_n)$ has the same law as $\rR_r$. So
\begin{align*}
\p{\max_{u\in v(\rR_r)} \deg_{\rR_r}(u)=x}
=&~ \p{\max\limits_{v\in v(\rR'(\rQ_n))}\deg_{\rR'(\rQ_n)}(v)=x~\Big\vert~ |v(\rR'(\rQ_n))|=r}\\
\le &~\frac{\p{\max\limits_{v\in v(\rR'(\rQ_n))} \deg_{\rR'(\rQ_n)}(v)=x}}{\p{|v(\rR'(\rQ_n))|=r}}
\le \frac{\p{\max\limits_{v\in v(\rQ_n)} \deg_{\rQ_n}(v)=x}}{\p{|v(\rR'(\rQ_n))|=r}}.
\end{align*}
Note that $n = \lfloor 15r/7 \rfloor$. By \citep[Proposition 4.3]{ABW}, there thus exists $c'>0$ such that $\p{|v(\rR'(\rQ_n))|=r}\ge  c' r^{-2/3}$. Together with (\ref{in:degbd}), it follows that for all $n\in\N$,
\[
\p{\max_{u\in v(\rR_r)} \deg_{\rR_r}(u)=x}\le \frac{c}{c'}\veps^x r^{2/3}n \le \frac{15c}{7c'}\veps^x r^{5/3}.\qedhere
\]
\end{proof}

Fix $\rQ_n\in \cQ_{n,r(n)}$ for now. List the edges of $R(\rQ_n)$ as $e_1,\ldots,e_{2r(n)-4}$ in $\prec_{\rQ_n}$-order. Create an extra copy $e_0$ of $e_1$ as in the decomposition described in the proof of \refL{quadcount}. For $i=0,\ldots,2r(n)-4$, write $\ell_i(\rQ_n)$ for the number of copies of $e_i$ in $\rQ_n$ minus one, that is, $\ell_i(\rQ_n)$ is the number of facial $2$-cycles resulting from the split of $e_i$.

\begin{lem}
\label{lem:split}
For $\rQ_n\in_u \cQ_{n,r(n)}$, $\p{\max_{0\le i\le 2r(n)-4}\ell_i(\rQ_n)>5 \ln r(n)} = O\left(r(n)^{-1}\right)$. 
\end{lem}

\begin{proof}
For $n\in\N$ let $\rQ_n\in_u\cQ_{n,r(n)}$. Note that $\sum_{i=0}^{2r(n)-4} \ell_i(\rQ_n) = |\cF(\rQ_n)|$. It follows that the vector $(\ell_0(\rQ_n),\ldots,\ell_{ 2r(n)-4})$ is distributed as a uniformly random weak composition of $|\cF(\rQ_n)|$ into $2r(n)-3$ parts. (Recall that in a weak composition, empty parts are allowed.) In particular, $\ell_0(\rQ_n)$ is distributed as the size of the first part in such a composition. Using that the number of weak compositions of $a$ into $b$ is ${a+b-1\choose b-1}$, and noting that $\left(1-\frac{b-1}{j+b-1}\right)$ is increasing in $j$, it follows that for integer $j\le k$, $\p{\ell_0(\rQ_n)>j~\big\vert~ |\cF(\rQ_n)|=k}
\le \left(1- \frac{2r(n)-4}{k+2r(n)-4}\right)^j$. Moreover, it follows from \refP{prop:Nsize} that $\p{ |\cF(\rQ_n)| \ge 3r(n)} = O\left(\left(4/9\right)^{r(n)} n^{5/2}\right)$. Recalling that $r(n)> (\ln n)^{25}$, it follows that
\begin{align*}
&\p{\ell_0(\rQ_n)> 5\ln r(n)}\\
\le&~
\sum_{k<3r(n)} \p{ \ell_0(\rQ_n)> 5 \ln r(n)~\big\vert~
 |\cF(\rQ_n)|=k}\p{ |\cF(\rQ_n)|=k}  +\p{ |\cF(\rQ_n)|\ge 3r(n)}\\
\le&~\left(1-\frac{2r(n)-4}{3r(n)+2r(n)-4}\right)^{5\ln r(n)}(1-o(1))+O\left(\left(4/9\right)^{r(n)} n^{5/2}\right)\\
\le &~e^{-2\ln r(n)}(1-o(1))+O\left(\left(4/9\right)^{r(n)} n^{5/2}\right)
=O\left(r(n)^{-2}\right).
\end{align*}
Finally, by a union bound,
\[
\p{\max_{0\le i\le 2r(n)-4}\ell_i(\rQ_n)> 5\ln r(n)}
\le 2r(n) \p{\ell_0(\rQ_n)>5\ln r(n)}=O\left(r(n)^{-1}\right).\qedhere
\]
\end{proof}

The next two facts provide deterministic bounds on the pointed GHP distance. Versions of these facts which apply to the non-pointed GHP distance appear in \citep[Facts 6.3 and 6.4]{ABW}, and we omit their proofs.

\begin{fact}\label{fact:ghp0}
Fix a pointed metric measure space $\bV=(V,d,o,\mu)$, and let $W \subset V$ with $o\in W$. Let $\mu_W$ be a Borel measure on $(W,d)$, and write $\bW=(W,d,o,\mu_W)$. Then 
\[
\dghp^\star(\bV,\bW)\le \max\left\{d_\rH(\bV,\bW),d_\rP(\mu,\mu_W)\right\}.
\]
\end{fact}

\begin{fact}\label{fact:ghp}
Fix a pointed metric measure space $\bV=(V,d,o,\mu)$. Let $W\subset V$ be finite with $o\in W$ so that there exists $\veps>0$ with $V= \{u\in V: d(u,W)\le\veps \}$. Let $\{P_w:w\in W\}$ be such that $\bigcup_{w\in W}P_w =V$, that $\mu(P_w \cap P_{w'})=0$ for $w \ne w'$, and that $P_w \subset \{u\in V: d(u,w)\le \veps\}$ for all $w \in W$. Define a measure $\nu$ on $W$ by setting $\nu(w) = \mu(P_w)$ for any $w\in W$, and let $\bW=(W,d,o,\nu)$. Then $\dghp^\star(\bV,\bW) \le \veps$.
\end{fact}

The final ingredient for proving \refP{prop:uan} is an asymptotic bound on the Prokhorov distance between the uniform measure on the vertices of a graph and a certain exchangeable perturbation of this measure. This is a reprise of \citep[Lemma 5.3, Corollaries 6.1 and 6.2]{ABW}. We start by introducing notations.

In the sequel, for $\bn = (\rn_1,\ldots,\rn_n)\in \R^n$ and $p>0$, write $|\bn|_p = \left(\sum_{i=1}^n \rn_i^p\right)^{1/p}$. Now, fix $n\in\N$, and let $\bn = (\rn_1,\ldots,\rn_n)$ be a vector of non-negative real numbers with $|\bn|_1>0$. Fix a rooted graph $\rG\in \cG_n$, and list the vertices as $v_1,\ldots, v_n$ in the $<_{\rG}$-order. Then define a measure on $v(\rG)$ by setting, for $V\subset v(\rG)$, $\mu_{\rG}^{\bn}(V) = \sum_{\{i:v_i \in V\}} \rn_i$. In words, we view $\rn_i$ as the total mass of pendant submaps attached to $v_i$, and $\mu_{\rG}^{\bn}$ as the measure assigning each vertex $v_i$ a mass of $\rn_i$. Recall that $\mu_\rG= \sum_{v\in v(\rG)} \delta_v$ is the counting measure.\footnote{Notice the different notations from \cite{ABW}, where the measures are defined with renormalization.}

\begin{lem}\label{lem:prok}
For $r\in\N$, let $\bn = (\rn_1,\ldots,\rn_r)$ be an exchangeable random vector of non-negative real numbers, and let $\rR_r\in_u\cR_r$. If $|\bn|_1\to\infty$ and $\frac{|\bn|_2}{|\bn|_1}\to0$ as $r\to\infty$, then $d_\rP\left(\frac{1}{|\bn|_1}\cdot \mu_{\rR_r}^{\bn},~\frac{1}{r}\cdot \mu_{\rR_r}\right) = o\left(r^{1/4}\right)$ with $1-o(1)$ probability, where $d_\rP$ is the Prokhorov distance on $\rR_r$.
\end{lem}

\begin{proof}
Fix $r\in\N$, and write $\rR = \rR_r$, for readability. List the vertices of $\rR$ as $v_1,\ldots,v_r$ in the $<_\rR$-order. It suffices to show that, for any $V\subset v(\rR)$ and for any $t>0$,
\begin{equation}\label{in:prok}
\p{\Big\vert \frac{1}{|\bn|_1}\cdot \mu_{\rR}^\bn(V) - \frac{1}{r}\cdot \mu_{\rR}(V) \Big\vert > \frac{2t}{|\bn|_1}~\bigg\vert~|\bn|_2} \le 2 \exp\left(-\frac{2t^2}{|\bn|_2^2}\right).
\end{equation}
Assuming that (\ref{in:prok}) holds, \refL{lem:prok} follows in a similar way as \citep[Corollary 7.2]{ABW} follows from \citep[Lemma 6.3 and Corollary 7.1]{ABW}. Now we turn to proving (\ref{in:prok}). Note that $\mu_{\rR}(V) = |V|$, and that $
\E{\sum_{\{i:v_i\in V\}} \rn_i~\Big\vert~|\bn|_1}
 = |\bn|_1\cdot \frac{|V|}{r}$. Then by a Hoeffding-type bound (see \citep[Theorem 2.5]{M}),
\begin{align*}
\p{\Big\vert \frac{1}{|\bn|_1}\cdot \mu_{\rR}^\bn(V) - \frac{1}{r}\cdot \mu_{\rR}(V) \Big\vert > \frac{2t}{|\bn|_1}~\bigg\vert~|\bn|_2}
=&~ \p{\Big\vert \sum_{\{i:v_i\in V\}} \rn_i -|\bn|_1\cdot \frac{|V|}{r} \Big\vert >2t~\bigg\vert~|\bn|_2}\\
\le&~ 2 \exp\left( - \frac{2t^2}{|\bn|_2^2} \right).\qedhere
\end{align*}
\end{proof}

\begin{proof}[{\bf Proof of \refP{prop:uan}}]
Fix $\rQ_n\in_u\cQ_{n,r(n)}$. Write $d_n$ for the distance on any induced submap of $\rQ_n$, let $R^+_n=R^+(\rQ_n)$ and $R_n = R(\rQ_n)$. Root $R_n$ at the $\prec_{\rQ_n}$-minimal edge, and write the resulting rooted map as $\rR_n$. Then list the vertices of $\rR_n$ as $v_1,\ldots,v_{r(n)}$ in the $<_{\rR_n}$-order, noting that $|v(R_n)|=r(n)$. Let $\cO_n$ be the set of components of $R^+_n - v(R_n)$. For each $v\in v(R_n)$, let $
C_v=\bigcup\left\{v(G):G\in \cO_n,d_{n}(G,v)=1\right\} ~\bigcup~\{v\}$ and $\bn =\left(|v(C_{v_i})|: 1\le i\le r(n)\right)$. Note that $\mu_{\rR_n}^{\bn}(v) = \mu_{R_n^+}(C_v)$ for $v\in v(R_n)$. Then let
 \[
 \overline{\bR}_n=\left(R_n,~\frac{1}{k_n}\cdot d_n,~\rho_{\rQ_n},~\frac{1}{|v(R_n^+)|}\cdot \mu_{\rR_n}^{\bn}\right),~ \widehat{\bR}_n^+=\left(R_n^+,~\frac{1}{k_n}\cdot d_n,~\rho_{\rQ_n},~\frac{1}{|v(R_n^+)|}\cdot \mu_{R_n^+}\right).
 \]
It follows from \refFt{fact:ghp} that $\dghp^\star\left(\overline{\bR}_n, \widehat{\bR}_n^+ \right)
\le \frac{1}{k_n} \max_{G\in\cO_n}(\diam(G)+1)$. By \refC{diam} and by the fact that $k_n=\Theta(r(n)^{1/4})$, $\frac{1}{k_n} \max_{G\in\cO_n}\diam(G)\convp0$. Hence, 
\begin{equation}\label{ghpbd1}
\dghp^\star\left( \overline{\bR}_n, \widehat{\bR}_n^+ \right)\convp0.
\end{equation}

Furthermore, we claim that
\begin{equation}\label{eq:claim}
\max_{v\in v(R_n)}|C_v|=o(r(n))
\end{equation}
 with $1-o(1)$ probability; this claim is proven in the end of this proof. Note that $|\bn|_1 = \sum_{v\in v(R_n)} |C_v| > r(n)$, and that $r(n)\to\infty$ with $n$. Then (\ref{eq:claim}) leads to $\frac{|\bn|_2}{|\bn|_1}= \frac{\left(\sum_{v\in v(R_n)} |C_v|^2\right)^{1/2}}{\sum_{v\in v(R_n)} |C_v|} \le \left(\frac{\max_{v\in v(R_n)} |C_v|}{\sum_{v\in v(R_n)} |C_v|}\right)^{1/2} = o(1)$ with $1-o(1)$ probability. This verifies the assumptions of \refL{lem:prok}. It follows from \refL{lem:prok} that $ d_\rP\left(\frac{1}{|\bn|_1}\cdot \mu_{\rR_n}^{\bn},~ \frac{1}{r(n)}\cdot \mu_{R_n}\right) = o\left(r(n)^{1/4}\right) = o\left(k_n\right)$ with $1-o(1)$ probability, where $d_\rP$ is the Prokhorov distance on $R_n$. Next, let $\widehat{\bR}_n =\left(R_n,~\frac{1}{k_n}\cdot d_n,~\rho_{\rQ_n},~\frac{1}{r(n)}\cdot \mu_{R_n}\right)$. Note that $\overline{\bR}_n$ and $\widehat{\bR}_n$ have the same metric structure but with different measures, and that $|v(R^+_n)|=|\bn|_1$. Then by \refFt{fact:ghp0},
\begin{equation}\label{ghpbd2}
\dghp^\star\left(\overline{\bR}_n,\widehat{\bR}_n\right) 
\le \frac{1}{k_n}\cdot  d_\rP\left(\frac{1}{|\bn|_1}\cdot \mu_{\rR_n}^{\bn},~\frac{1}{ r(n)}\cdot \mu_{R_n}\right)
\convp0.
\end{equation}
As noted in \refS{sec:outline}, $\widehat{\bR}_n\convdist\bm_\infty$ for the pointed GHP topology. Combined with (\ref{ghpbd1}) and (\ref{ghpbd2}), we thus have $\widehat{\bR}_n^+ \convdist \bm_\infty$ for the pointed GHP topology, establishing (\ref{conv:uan}).

To prove (\ref{eq:claim}), note $
\max_{v\in v(R_n)}|C_v|\le \max_{0\le i\le 2r(n)-4}\ell_i(\rQ_n)  \cdot \max_{v\in v(R_n)}\deg_{R_n}(v) \cdot Y_{(2)}(\rQ_n)+1$. By \refC{Rcor2}, $Y_{(2)}(\rQ_n) \le r(n)^{5/6}$ with $1-o(1)$ probability. Moreover, it follows from \refL{lem:vtxdegree} that $\p{\max_{v\in v(R_n)}\deg_{R_n}(v)\ge 3\ln r(n)} =O\left(e^{-3\ln r(n)} r(n)^{5/3}\right)=O(r(n)^{-1})$. Finally, by \refL{lem:split}, $ \max_{0\le i\le 2r(n)-4}\ell_i(\rQ_n) \le 5 \ln r(n)$ with $1-o(1)$ probability. Together with the previous bounds, $\max\limits_{v\in v(R_n)}|C_v|= O\left((\ln r(n))^2r(n)^{5/6}\right) = o(r(n))$ with $1-o(1)$ probability, establishing (\ref{eq:claim}).
\end{proof}

\section{Proofs of the Main Theorems}\label{sec:pfthm1}

Given pointed metric measure spaces $\bX=(X,d,x,\mu)$ and $\bY=(Y,d',y,\mu')$, let $Z = (X\setminus\{x\})\cup Y$, and define a distance $\delta$ on $Z$ by setting, for $p,q\in Z$,
\[
\delta(p,q) = \begin{cases} d(p,q) & \mbox{ if } p,q\in X\\
d'(p,q) & \mbox{ if } p,q \in Y\\
d(p,x) + d'(y,q) &\mbox{ if } p\in X, q\in Y
\end{cases}.
\]
Then define a measure $\nu$ on the Borel sets of $(Z,\delta)$ by setting $\nu(V) = \mu(V\cap X\setminus\{x\}) + \mu'(V\cap Y)$. Let $\bZ(\bX,\bY)= (Z,\delta,y,\nu)$. In words, $\bZ(\bX,\bY)$ is the pointed metric measure space obtained from $\bX$ and $\bY$ by identifying the distinguished points of $\bX$ and $\bY$.

Recall that given a pointed metric measure space $\bV=(V,d,o,\nu)$, $B_r=B_r(\bV) = \{w\in V: d(w,o)\le r\}$, and $\bB_r(\bV) = \left(B_r,d,o,\nu\big\vert_{B_r}\right)$. The {\em local GHP distance} between two pointed metric measure spaces $\bV=(V,d,v,\mu)$ and $\bW=(W,d',w,\mu')$ is $\dlghp\left(\bV,\bW \right)
= \sum_{r=1}^\infty \frac{\min\left\{\dghp^\star\left(\bB_r(\bV),\bB_r(\bW)\right),1\right\}}{2^r}$.

\begin{lem}\label{localgh}
Given pointed metric measure spaces $(\bX_n:1\le n\le \infty),(\bY_n:1\le n\le \infty)$, if $\dghp^\star\left(\bX_n,\bX_\infty\right)\to0$ and $\dlghp\left(\bY_n,\bY_\infty\right)\to0$, $\dlghp\left(\bZ(\bX_n,\bY_n),\bZ(\bX_\infty,
\bY_\infty)\right)\to0 $.
\end{lem}

\begin{proof}
Write $\bZ_n = \bZ(\bX_n,\bY_n)$ and $\bZ_\infty =\bZ(\bX_\infty,\bY_\infty)$. Let $r\ge0$. For $n\in\N$, we have $
\dghp^\star\left(\bB_r(\bZ_n),\bB_r(\bZ_\infty)\right)
\le \dghp^\star\left(\bB_r(\bX_n),\bB_r(\bX_\infty)\right) + \dghp^\star\left(\bB_r(\bY_n),
\bB_r(\bY_\infty)\right)$.
By assumption, the right hand side tends to $0$.
\end{proof}

As discussed in \refS{sec:main}, all graphs are endowed with edge lengths and viewed as length spaces. Recall that 
$k_n = \left(\frac{40\cdot r(n)}{21}\right)^{1/4}$, and the following notations from Sections~\ref{sec:graphintro} and~\ref{sec:outline}: given $\rQ\in\cQ_n$, $L=P_{(1)}(\rQ)$ is the largest submap pendant to $R(\rQ)$, $\widehat{\bL}(\rQ) = \left(L,~\frac{1}{k_n}\cdot d_L,~\rho_{\rQ},~\frac{8}{9k_n^4}\cdot \mu_L\right)$, $R^+=R^+(\rQ)= \rQ -( v(L)\setminus\{\rho_{\rQ}\})$, and $\widehat{\bR}^+(\rQ)= \left(R^+,~\frac{1}{k_n}\cdot d_{R^+},~\rho_\rQ,~\frac{1}{|v(R^+)|}\cdot \mu_{R^+}\right)$. The following lemma relies on \refP{planeghp}.

\begin{lem}\label{localgh0}
For $\rQ_n\in_u\cQ_{n,r(n)}$, $
\left(\widehat{\bR}^+(\rQ_n),\widehat{\bL}(\rQ_n)\right) \convdist (\bm_\infty,\mathbfcal{P})$ for the local GHP topology, where $\bm_\infty$ and $\mathbfcal{P}$ are independent.
\end{lem}

\begin{proof}
By \refP{prop:uan}, $\widehat{\bR}^+(\rQ_n)\convdist \bm_\infty$ for the pointed GHP topology, and it is easily seen that the convergence also holds for the local GHP topology. Moreover, we show in \refP{planeghp} that $\widehat{\bL}(\rQ_n)\convdist \mathbfcal{P}$ for the local GHP topology. Finally, the independence between $\bm_\infty$ and $\mathbfcal{P}$ follows from the conditional independence of $R^+(\rQ_n)\setminus\{\rho_{\rQ_n}\}$ and $P_{(1)}(\rQ_n)$ given their sizes.
\end{proof}

\begin{proof}[{\bf Proof of Theorem~\ref{thm2}}]
It follows from \refL{localgh0} and the Skorokhod representation theorem that there exists a probability space where $\left(\widehat{\bR}^+(\rQ_n),\widehat{\bL}(\rQ_n)\right) \to (\bm_\infty,\mathbfcal{P})$ almost surely. \refL{localgh} then yields that in this space we have $\bZ\left(\widehat{\bR}^+(\rQ_n),\widehat{\bL}(\rQ_n)\right)  \to \bZ\left(\bm_\infty,\mathbfcal{P}\right)$ almost surely, which implies convergence in distribution. It is easily seen that 
\begin{align*}
&\bZ\left(\widehat{\bR}^+(\rQ_n),\widehat{\bL}(\rQ_n)\right)
 \eqdist \left(\rQ_n,~ \frac{1}{k_n}\cdot d_{\rQ_n},\rho_{\rQ_n}, \frac{8}{9k_n^4}\cdot \mu_{L(\rQ_n)} + \frac{1}{|v(R^+(\rQ_n))|}\cdot (\mu_{R^{+}(\rQ_n)}-\delta_{\rho_{\rQ_n}})\right) .
\end{align*}
In the above, $(\mu_{R^{+}(\rQ_n)}-\delta_{\rho_{\rQ_n}})$ can be replaced by $\mu_{R^{+}(\rQ_n)}$ without affecting the convergence in distribution since $|v(R^+(\rQ_n))|\to\infty$. Finally, from the definitions of $\mathbfcal{S}$ and $\mathbfcal{P}$ given in \refApp{sec:minibus}, and the definition of $\bm_\infty$ in \refApp{sec:planeghp}, we have $\bZ\left(\bm_\infty,\mathbfcal{P}\right)\eqdist \mathbfcal{S}$. Briefly: the equivalence of metric structure is clear, and the measure of $\mathbfcal{S}$, defined as $(\pi_1\circ p_{(1)})_*\mathrm{Leb}_{[0,1]} + (\pi_\infty\circ p_\infty)_*\mathrm{Leb}_\R$, is equal to the measure of $\bZ\left(\bm_\infty,\mathbfcal{P}\right)$, since the point in $\bm_\infty$ which is glued to the distinguished point of $\mathbfcal{P}$ has measure $0$ almost surely. 
\end{proof}

Theorem~\ref{thm1} follows from \refP{prop:uan1} in the same way as Theorem~\ref{thm2} follows from \refP{prop:uan}, so we omit the proof.

\appendix

\section{The Brownian Plane, with and without Minbus}\label{sec:minibus}

In the remaining paper, for $s,t\in\R$, we write $s\wedge t = \min\{s,t\}$ and $s\vee t = \max\{s,t\}$.

\subsection{The Brownian Plane with Minbus}\label{sec:planeminbus}

The Brownian plane with minbus, $\mathbfcal{S}$, is the quotient space obtained from gluing the root of the Brownian plane to a random point of the Brownian map. We present a construction of $\mathbfcal{S}$ in this subsection, partly because it explicitly describes the infinite measure of $\mathbfcal{S}$.\footnote{This may also help comprehension (especially \refApp{sec:planeghp}) if the reader has not seen the construction of the Brownian plane/map.}

Let $\be = (\be_t)_{t\in[0, 1]}$ be a standard Brownian excursion. Define a process $Z' = (Z'_t)_{t\in[0,1]}$ such that, conditioned on $\be$, $Z'$ is a centred Gaussian process with covariance $\E{Z'_s Z'_t~|~\be} = \min_{r\in [s\wedge t,s\vee t]} \be_r
$ for any $s,t\in[0,1]$. Then shift the time index of the pair $(\be_t, Z'_t)_{t\in[0,1]}$ so that the ``new $Z'_0$" is minimal among $(Z'_t)_{t\in[0,1]}$. More precisely, by \citep[Proposition 2.5]{LGW}, there exists an almost surely unique time $s_*\in [0,1]$ such that $Z'_{s_*} = \min \{Z'_t: t\in [0,1]\}$. Now, for any $t\in [0,1]$, let $\bar{\be}_t = \be_{s_*} + \be_{s_* \oplus t} - 2\inf_{r\in [s_*\wedge s_* \oplus t,s_*\vee s_* \oplus t]} \be_r$, and $\bar{Z}'_t = Z'_{s_* \oplus t} - Z'_{s_*}$, where $s_* \oplus t = s + t$ if $s_* +t \le 1$, and $s_* \oplus t = s+t -1$ otherwise. By \cite[Theorem 1.2]{LGW}, $(\bar{\be}_t,\bar{Z}'_t)_{t\in [0,1]}$ has the same distribution as $(\be_t,Z'_t)_{t\in[0,1]}$ conditioned on $\min_{t\in [0,1]} Z'_t \ge 0$. The Continuum Random Tree (CRT) coded by $\bar{\be}$ may be viewed as the CRT coded by $\be$ re-rooted at the vertex with minimal label, and the labels $\bar{Z}'$ on the CRT coded by $\bar{\be}$ are derived from $Z'$ by subtracting the minimal label; see \citet[Section 2.3]{BLG}. 

Next, let $R=(R_t)_{t\ge 0}$ and $R'= (R'_t)_{t\ge 0}$ be two independent $3$-dimensional Bessel processes started from $0$, independent of $\be$. Define $\rR=(\rR_t)_{t\in\R}$ by setting
\begin{equation}\label{eq:bessel}
\rR_t = \begin{cases}R_{t} &\mbox{ if } t\ge 0\\
R'_{-t} &\mbox{ if } t<0 \end{cases}.
\end{equation}
Then for any $s,t\in\R$, let
\[
\overline{st} =\begin{cases} [s\wedge t, s\vee t] &\mbox{ if } st\ge 0\\
(-\infty,s\wedge t]\cup [s\vee t,\infty) &\mbox{ if }  st< 0\end{cases},
\]
and define a process $Z= (Z_t)_{t\in\R}$ such that, conditioned on $\rR$, $Z$ is the centred Gaussian process with covariance
\begin{equation}\label{eq:Z}
\E{Z_s Z_t~|~\rR} = \inf_{r\in \overline{st}} \rR_r, ~s,t\in\R.
\end{equation}
Let ${\rX} = ({\rX}_t)_{t\in\R}$ be a concatenation of $\bar{\be}$ and $\rR$ by setting
\[
{\rX}_t = \begin{cases} \bar{\be}_t &\mbox{ if } 0\le t\le 1\\
R_{t-1} &\mbox{ if } t>1\\
R'_{-t} &\mbox{ if } t<0\end{cases}. \mbox{ Similarly, set }
{\rW}_t = \begin{cases} \bar{Z}'_t &\mbox{ if } 0\le t\le 1\\
Z_{t-1} &\mbox{ if } t>1 \\
Z_{t} &\mbox{ if } t<0\end{cases}.
\]
For any $s,t\in\R$, define
\[
\widehat{st} = \begin{cases}
(-\infty, s\wedge t] \cup [s \vee t,\infty) &\mbox{ if } st<0 \mbox{ and } 
s\vee t>1\\
[s \wedge t,s \vee t] &\mbox{ otherwise }
\end{cases}.
\]
Now, we define a random pseudo-metric $d_{\rX}$ on $\R^2$ by setting, for any $s,t\in\R$, $d_{\rX}(s,t) = \rX_s +\rX_t - 2 \inf_{r\in \widehat{st}} \rX_r$. Write $s\sim_{\rX} t$ if $d_{\rX}(s,t)=0$, and let $\cT=\R/\sim_{\rX}$. Informally, we may view $\cT$ as obtained from gluing the root of Aldous' CRT at the root of infinite Brownian tree. It is easily seen that $\rW_0=0$, $\E{(\rW_s - \rW_t)^2~|~\rX} = d_\rX(s,t)$, and $\rW$ has a modification with continuous paths (we shall view $\rW$ as such in the sequel). Then $d_\rX(s,t)=0$ implies $\rW_s = \rW_t$ almost surely, so we may view $\rW$ as indexed by $\cT$.

Furthermore, for any $s,t\in\R$, let
\begin{equation}\label{eq:Dcirc}
D^\circ(s,t) =\rW_s +\rW_t - 2\inf_{r\in[s\wedge t, s\vee t]} \rW_r.
\end{equation}
Write $p: \R\to{\cT}$ for the canonical projection, then we extend the definition of $D^\circ$ to ${\cT} \times {\cT}$ by setting, for any $a,b\in{\cT}$, $
D^\circ(a,b) = \min\left\{D^\circ(s,t): s,t\in\R, p(s) = a,p(t)=b\right\}$. Let
\begin{equation}\label{eq:D}
D(a,b) = \inf_{a_0=a,a_1,\dots,a_k=b} \sum_{i=1}^k D^\circ(a_{i-1},a_i)
\end{equation}
with the infimum taken over all choices of $k\in\N$ and of the finite sequence $a_0=a,a_1,\dots,a_k=b\in{\cT}$. It follows that $D$ is a pseudo-metric on ${\cT}$. Write $\cS = {\cT}/\{D=0\}$, and let $\rho\in\cS$ be the equivalence class of $p(0)$. Let $\pi$ be the canonical projection from $\cT$ to $\cS$, and we continue to use $D$ to denote the push-forward of $D$ by $\pi\times\pi$ to $\cS$.

Finally, write $\mathrm{Leb}_I$ for the Lebesgue measure over interval $I\subset \R$, and let $\mu= (\pi\circ p)_* \mathrm{Leb}_\R$. The {\em pointed measured Brownian plane with minbus} is the pointed metric measure space $\mathbfcal{S}\coloneqq (\cS,D,\rho,\mu)$.

\subsection{The Brownian Plane}

We quickly go over the definition of the Brownian plane from \cite{CLG}, referring the reader to that work for a full exposition.

Let $R$ and $R'$ be two independent $3$-dimensional Bessel processes started from $0$, and define $\rR$ as in (\ref{eq:bessel}). Define a random pseudo-metric $d_\rR$ on $\R^2$ by setting, for any $s,t\in\R$, $d_\rR(s,t) = \rR_s+ \rR_t - 2\inf_{r\in \overline{st}}\rR_r$. Write $s\sim_\rR t$ if $d_\rR(s,t)=0$. The quotient space $\cT_\infty \coloneqq \R/\sim_\rR$ equipped with $d_\rR$ is called the infinite Brownian tree. Conditionally given $\rR$, let $Z$ be the centred Gaussian process with covariance as in (\ref{eq:Z}). Then define $D^\circ_\infty$ similarly as $D^\circ$ in (\ref{eq:Dcirc}) with $\rW$ replaced by $Z$, and define $D_\infty$ analogously to $D$ in (\ref{eq:D}). 

Write $\cP = \cT / \{D_\infty=0\}$. Let $p_\infty:\R\to\cT_\infty$ and $\pi_\infty: \cT_\infty \to\cP$ be the canonical projections. Let $\rho_\infty\in \cP$ be the equivalence class of $p_\infty(0)$. Finally, let $\mu_\infty = (\pi_\infty \circ p_\infty)_* \mathrm{Leb}_\R$. Then write $\mathbfcal{P}= (\cP,D_\infty,\rho_\infty,\mu_\infty)$ for the pointed measured Brownian plane.

\section{Convergence to the Brownian Plane for the Local GHP Topology}\label{sec:planeghp}

In this section, we establish the convergence towards the Brownian plane in the GHP topology, extending the result of \cite{CLG} for the GH topology. 

\subsection{Scaled Brownian Map}

We elaborate a bit on the definition of the scaled Brownian map \cite{CLG}, to make this article more self-contained, but follow the notation of that paper.

Fix $\lambda>0$ in this subsection, and let $\be^\lambda=(\be^\lambda_t)_{t\in[0,\lambda^4]}$ be a Brownian excursion of lifetime $\lambda^4$. Write $\cT_{(\lambda)}$ for the scaled Brownian CRT indexed by $\be^\lambda$, and let $p_{(\lambda)}:[0,\lambda^4]\to \cT_{(\lambda)}$ be the canonical projection, sending $x\in[0,\lambda^4]$ to its equivalence class in $\cT_{(\lambda)}$. Conditionally given $\be^\lambda$, let $Z^\lambda = (Z^\lambda_t)_{0\le t\le \lambda^4}$ be the centred Gaussian process with covariance
\begin{equation}\label{eq:Zlambda}
\E{Z^\lambda_s Z^\lambda_t~\big\vert~\be^\lambda} = \min_{r\in [s\wedge t,s\vee t]} \be^\lambda_r.
\end{equation}

Furthermore, for $s,t\in[0,\lambda^4]$ with $s\le t$, we let $D_\lambda^\circ(s,t) = D_\lambda^\circ(t,s)= Z^\lambda_s + Z^\lambda_t - 2 \max\left\{
\min_{r\in[s,t]} Z^\lambda_r,
\min_{r\in [t,\lambda^4]\cup[0,s]} Z^\lambda_r
\right\}$. Now extend the definition of $D_\lambda^\circ$ to $\cT_{(\lambda)}\times \cT_{(\lambda)}$ by setting, for any $a,b\in\cT_{(\lambda)}$, $D_\lambda^\circ(a,b) = \min\left\{ D_\lambda^\circ(s,t):s,t\in [0,\lambda^4], p_{(\lambda)}(s) = a,p_{(\lambda)}(t)=b\right\}$, and $D^*_\lambda(a,b) = \inf_{a_0=a,a_1,\ldots,a_p=b}\sum_{i=1}^p D_\lambda^\circ(a_{i-1},a_i)$, where the infimum is over all choices of $p\in\N$ and of the finite sequence $a_0=a,a_1,\ldots,a_p=b$ in $\cT_{(\lambda)}$. It follows that $D^*_\lambda$ is a pseudo-metric on $\cT_{(\lambda)}$. Write $Y^\lambda = \cT_{(\lambda)}/\{D^*_\lambda = 0\}$, and let $\rho_\lambda\in Y^\lambda$ be the equivalence class in $Y^\lambda$ of $p_{(\lambda)}(0)$. Let $\pi_\lambda$ be the canonical projection from $\cT_{(\lambda)}$ to $Y^\lambda$, and we continue to use $D^*_\lambda$ to denote the push-forward of $D^*_\lambda$ by $\pi_\lambda\times \pi_\lambda$ to $Y^\lambda$. 

Finally, let $\mu_\lambda = (\pi_\lambda\circ p_{(\lambda)})_*\mathrm{Leb}_{[0,\lambda^4]}$. The {\em pointed measured scaled Brownian map} is $\bY^\lambda \coloneqq(Y^\lambda,D^*_\lambda,\rho_\lambda,\mu_\lambda)$. Taking $\lambda=1$, $\cT_{(1)}$ is the Brownian CRT. Write $\bm_\infty = (Y^1,D^*_1,\rho_1,\mu_1)$ for the pointed measured Brownian map. For any $\lambda>0$, write $
\lambda\cdot \bm_\infty = (Y^1,~\lambda\cdot D^*_1,~\rho_1,~\lambda^4\cdot \mu_1)$.

\subsection{A Nice Event}\label{sec:event}

\citep[Proposition 4]{CLG} defines an event on which, $\lambda\cdot \bm_\infty$ and $\mathbfcal{P}$ have the same {\em local metric} structure. In \refP{planeghp} below, we show that on this event, $\lambda\cdot \bm_\infty$ and $\mathbfcal{P}$ also have the same local structure with their endowed measures. The purpose of the current subsection is to describe this event.

Fix $A>1$, $\alpha>0$ and $\lambda>(2\alpha)^{1/4}$. Let $\be^\lambda$ be a copy of Brownian excursion of lifetime $\lambda^4$, and let $\rR=(R,R')$ be copies of independent $3$-dimensional Bessel processes. Next, let $Z$ and $Z^\lambda$ be centred Gaussian processes with covariances, respectively, given in (\ref{eq:Z}) and (\ref{eq:Zlambda}). Furthermore, for every $x\ge0$, let $\gamma_\infty(x) = \sup\{t\ge0:R_t=x\}$.  Now define
\[
\cE_\lambda=\cE_{\alpha,\lambda}(\be^\lambda,R,R') = \left\{\be_t^\lambda = R_t \mbox{ and }
\be^\lambda_{\lambda^4-t} = R'_t, \forall t\le \alpha\right\} \cap \left\{ \min_{\alpha\le t\le \lambda^4-\alpha} \be_t^\lambda
=\inf_{t\ge \alpha} R_t \wedge \inf_{t\ge \alpha} R_t'\right\}.
\]
As in the proof of \citep[Proposition 4]{CLG}, on $\cE_\lambda$ we have $Z^\lambda_t = Z_t,~ Z^\lambda_{\lambda^4-t} = Z_{-t}, ~\forall t\in [0,\alpha]$. Then let $\cF_\lambda=\cF_{A,\alpha,\lambda}(\be^\lambda,R,R',Z^\lambda,Z)$ be the intersection of $\cE_\lambda$ with the following events: $\inf_{t\ge \alpha} R_t \wedge \inf_{t\ge \alpha} R_t' > A^4$, $\min_{0\le x\le A} Z_{\gamma_\infty(x)}<-10,~ \min_{A\le x\le A^2} Z_{\gamma_\infty(x)}<-10$, and $\min_{A^2\le x\le A^4}Z_{\gamma_\infty(x)}<-10$.

\subsection{Convergence to the Brownian Plane}

Recall that given $\rQ=(Q,e)\in\cQ$, $\mu_{Q} = \sum_{v\in v(Q)}\delta_v$. Since local GHP convergence is only stated for length spaces, we view each edge $e$ of $Q$ as an isometric copy of the unit interval $[0,1]$. We abuse notation and continue to write $(Q,d_Q)$ for the resulting length space. In this appendix, for $c>0$ write $c\cdot \bQ = \left(Q,~c\cdot d_Q,~u,~\frac{8c^4}{9}\cdot \mu_Q\right)$, where $u$ is the tail of the root edge $e$.
 
\begin{prop}\label{planeghp}
Let $(k_n\in\R_+:n\in\N)$ be such that $k_n\to\infty$ and $k_n=o(n^{1/4})$. Then for $\rQ_n\in_u\cQ_n$, $
k_n^{-1}\cdot  \bQ_n \convdist \mathbfcal{P}$ for the local GHP topology.
\end{prop}

Recall that given a pointed metric measure space $\bV=(V,d,o,\nu)$, writing $B_r = B_r(\bV)$, $\rB_r(\bV) = \left(B_r,~d,~o\right)$, and $\bB_r(\bV) = \left(B_r,~d,~o,~\nu\big\vert_{B_r}\right)$, where $\nu\big\vert_{B_r}$ denotes the measure $\nu$ restricted to $B_r$. It suffices to show that, given $\rQ_n\in_u\cQ_n$, for any $r\ge0$,
\begin{equation}
\label{conv:ballghp}
\bB_r( k_n^{-1}\cdot \bQ_n) \convdist \bB_r(\mathbfcal{P})
\end{equation}
for the pointed GHP topology. We will show the convergence for $r=1$, for ease of notation, and the argument for $r\neq 1$ follows similarly. 

\begin{proof}[{\bf Proof of \refP{planeghp}}]
This proof is a slight extension of that for \citep[Theorem 1.2]{CLG}.

Fix $\veps>0$. It follows immediately from \citep[Proposition 3]{CLG} and the proof of \citep[Proposition 4]{CLG} that there exist $A>1$, $\alpha>0$, and $\lambda_0>(2\alpha)^{1/4}$ such that for all $\lambda\ge \lambda_0$, we can construct copies of $\be^\lambda$, $R$, $R'$, $Z^\lambda$, $Z$, $\lambda\cdot \bm_\infty$, and $\mathbfcal{P}$ on a common probability space in such a way that, with probability at least $1-\veps$, the event $\cF_\lambda = \cF_{A,\alpha,\lambda}(\be^\lambda,R,R',Z^\lambda,Z)$ holds. As shown in \citep[Proposition 4]{CLG}, on the event $\cF_\lambda$, it holds that $\rB_1(\lambda\cdot \bm_\infty) = \rB_1(\mathbfcal{P})$.

On the other hand, by \citep[Proposition 9]{CLG}, there exists $\alpha_0=\alpha_0(\veps)>0$ such that, for every sufficiently large integers $m$ and $n$ with $n>m$, we can construct $\rQ_n\in_u\cQ_n$ and $\rQ_m\in_u\cQ_m$ on a common probability space in such a way that the equality $\rB_{\alpha_0 m^{1/4}} (\bQ_n) = \rB_{\alpha_0 m^{1/4}}(\bQ_m)$ holds with probability at least $1-\veps$. 

Without loss of generality, we assume that $\alpha_0 < \frac{1}{2\lambda_0}$ and $k_n \le \alpha_0 \lfloor n^{1/4} \rfloor$ for all $n\in\N$. Write $\lambda = \alpha_0^{-1} \left(\frac{8}{9}\right)^{1/4}$, and note that $\lambda>\lambda_0$. For $n\in\N$, let $ m_n = \lceil \alpha_0^{-1} k_n\rceil^4$. Since $m_n$ tends to infinity with $n$, it follows that for large enough $n$, we may couple $\rQ_n$ and $\rQ_{m_n}$ such that the equality $\rB_1(k_n^{-1}\cdot \bQ_n) = \rB_1(k_n^{-1}\cdot \bQ_{m_n})$ holds with probability at least $1-\veps$. Since $\nu_{\rQ_n}$ and $\nu_{\rQ_{m_n}}$ both are counting measures, it follows from the previous equality that, with probability at least $1-\veps$, $\nu_{\rQ_n}\big\vert_{\rB_1(k_n^{-1}\cdot \bQ_n)} = \nu_{\rQ_{m_n}}\big\vert_{\rB_1(k_n^{-1}\cdot \bQ_{m_n})}$.

In the remainder of the proof, let $T = \lambda^4$. For every $x\in [0,\be^\lambda_{T/2}]$, set $\gamma_\lambda(x) = \sup\left\{t\le T/2: \be_t^\lambda=x\right\}$ and $\eta_\lambda(x) = \inf\left\{t\ge  T/2: \be_t^\lambda=x\right\}$. By \citep[Lemma 5]{CLG}, on the event $\cF_\lambda$, if $D^*_\lambda(\rho_\lambda,p_{(\lambda)}(t))\le 1$ then $t
\in [0,\gamma_\lambda(A))\cup(\eta_\lambda(A),T]$. From the proof of \citep[Proposition 4]{CLG}, we also know that $\gamma_\lambda(A)<\alpha$ and $T-\eta_\lambda(A)<\alpha$. Recalling the definition of $\cF_\lambda$ from \refS{sec:event}, it follows that on $\cF_\lambda$, we simultaneously have $\be^\lambda_t = R_t,~Z_t^\lambda = Z_t, ~\forall t\in[0,\gamma_\lambda(A)]$, and
$
 \be^\lambda_{t}=R_{T-t}', ~Z_t^\lambda = Z_{t-T},~\forall t\in[\eta_\lambda(A),T]$. This implies that on $\cF_\lambda$, $\pi_\lambda\circ p_{(\lambda)}\big\vert_{[0,\gamma_\lambda(A))\cup(\eta_\lambda(A),T]} = \pi_\infty \circ p_\infty\big\vert_{[0,\gamma_\lambda(A))\cup(\eta_\lambda(A)-T,0]}$, where $\pi_\lambda\circ p_{(\lambda)}$ is the canonical projection from $[0,T]$ to $\lambda\cdot \bm_\infty$, and $\pi_\infty \circ p_\infty $ is the canonical projection from $\R$ to $\mathbfcal{P}$. Since $B_1(\lambda\cdot \bm_\infty)\subset \pi_\lambda\circ p_{(\lambda)}\left([0,\gamma_\lambda(A))\cup(\eta_\lambda(A),T]\right)$ and $B_1(\mathbfcal{P})\subset \pi_\infty \circ p_\infty\left([0,\gamma_\lambda(A))\cup(\eta_\lambda(A)-T,0]\right)$, we obtain that, on $\cF_\lambda$, the measured versions of $\rB_1(\lambda\cdot \bm_\infty)$ and $\rB_1(\mathbfcal{P})$ are also equal: $\bB_1(\lambda\cdot \bm_\infty) = \bB_1(\mathbfcal{P})$.

Next, since $m_n = \lceil \lambda k_n (9/8)^{1/4}\rceil^4$ for all $n\in\N$, by \citep[Theorem 1]{Mi} and \citep[Theorem 1.1]{LG}, $k_n^{-1} \cdot \bQ_{m_n} \convdist \lambda \cdot \bm_\infty$ for the pointed GHP topology. (In \cite{CLG,Mi}, the convergence is only stated for the GH topology, but the proof in fact yields the above formulation. This is also stated explicitly in \citep[Theorem 4.1]{ABA}.) So, for the pointed GHP topology,
\begin{equation}\label{conv:ghpstar}
\bB_1(k_n^{-1}\cdot \bQ_{m_n})\convdist \bB_1(\lambda\cdot  \bm_\infty).
\end{equation}

Finally, it follows that we may simultaneously couple $\rQ_n$, $\rQ_{m_n}$, $\lambda\cdot \bm_\infty$, and $\mathbfcal{P}$ so that with probability at least $1-2\veps$ we have both $\bB_1(k_n^{-1}\cdot \bQ_n) = \bB_1(k_n^{-1}\cdot \bQ_{m_n})$ and $\bB_1(\lambda\cdot  \bm_\infty) = \bB_1(\mathbfcal{P})$. Write $\K^\star$ for the set of isometry-equivalence classes of pointed compact metric measure spaces. In a space where such a coupling holds, for any bounded continuous function $F:\K^\star \to \R$, we have $
\left\vert\E{ F(\bB_1(k_n^{-1}\cdot \bQ_n)) - F(\bB_1(\mathbfcal{P})) }\right\vert 
\le
\left\vert\E{  F(\bB_1(k_n^{-1}\cdot \bQ_n)) - F(\bB_1(k_n^{-1}\cdot \bQ_{m_n})) }\right\vert + \left\vert\E{  F(\bB_1(k_n^{-1}\cdot \bQ_{m_n})) 
- F(\bB_1(\lambda\cdot  \bm_\infty))}\right\vert
 +\left\vert\E{ F(\bB_1(\lambda \cdot \bm_\infty)) 
- F(\bB_1(\mathbfcal{P}))}\right\vert$. Writing $\|F\| = \sup_{x\in \K^\star} F(x)$, the first and the third terms on the right of the inequality are each less than $2\veps \|F\|$. The second term tends to $0$ with $n$, by (\ref{conv:ghpstar}). Therefore, $
\limsup_{n\to\infty} \left\vert \E{F(\bB_1(k_n^{-1}\cdot \bQ_n))}
- \E{F(\bB_1(\mathbfcal{P}))}\right\vert < 4\veps \|F\|$. Since $\veps$ was arbitrary, it follows that $ \E{F(\bB_1(k_n^{-1}\cdot \bQ_n))} \to \E{F(\bB_1(\mathbfcal{P}))}$, so $\bB_1(k_n^{-1}\cdot \bQ_n)\convdist \bB_1(\mathbfcal{P})$ for the pointed GHP topology by the Portmanteau theorem. As noted above, the case $r\neq1$ of (\ref{conv:ballghp}) follows by a similar argument.
\end{proof}

\end{document}